\documentclass{amsart}
\usepackage{amsmath,amsthm,amssymb,mathtools,amscd,color,comment}
\usepackage{cleveref}
\usepackage{aliascnt}
\usepackage{marginnote}
\usepackage{xcolor}

\usepackage[all]{xy}
\usepackage{enumerate}
\usepackage{url}
\usepackage{setspace}
\allowdisplaybreaks

\theoremstyle{plain}
\newtheorem{thm}{Theorem}[section]
\newaliascnt{prop}{thm}   
\newtheorem{prop}[prop]{Proposition}
\aliascntresetthe{prop}
\newaliascnt{lem}{thm}   
\newtheorem{lem}[lem]{Lemma}
\aliascntresetthe{lem}
\newaliascnt{cor}{thm}   

\aliascntresetthe{cor}
\newaliascnt{Pb}{thm}   
\newtheorem{Pb}[Pb]{Problem}
\aliascntresetthe{Pb}
\theoremstyle{definition}
\newaliascnt{dfn}{thm}   
\newtheorem{dfn}[dfn]{Definition}
\aliascntresetthe{dfn}
\newaliascnt{rem}{thm}   
\newtheorem{rem}[rem]{Remark}
\aliascntresetthe{rem}
\newaliascnt{ex}{thm}   
\newtheorem{ex}[ex]{Example}
\aliascntresetthe{ex}
\newaliascnt{conj}{thm}

\aliascntresetthe{conj}
\crefname{conj}{Conjecture}{Conjectures}
\Crefname{conj}{Conjecture}{Conjectures}

\DeclareMathOperator{\Harm}{Harm}

\DeclareMathOperator{\LCM}{lcm}
\DeclareMathOperator{\GCD}{gcd}

\let\Im\relax
\DeclareMathOperator{\Im}{Im}

\theoremstyle{definition}

\DeclareMathOperator{\Hst}{Hst}

\newcommand{\RR}{\mathbb{R}}
\newcommand{\CC}{\mathbb{C}}
\newcommand{\NN}{\mathbb{N}}
\newcommand{\Sd}{\mathbb{S}}
\newcommand{\la}{\langle}
\newcommand{\ra}{\rangle}

\makeatletter

\title{Spherical Designs with Infinite Harmonic Strength}

\author{Ryutaro Misawa}
\address{Graduate School of Information Sciences, Tohoku University}
\email{misawa.ryutaro.q2@dc.tohoku.ac.jp}

\author{Yusaku Nishimura}
\address{Graduate School of Fundamental Science and Engineering, Waseda University, Tokyo 169-8555, Japan}
\email{n2357y@ruri.waseda.jp}


\date{\today}

\begin{document}

\begin{abstract}

In this paper, we study the existence problem for spherical \(T\)-designs on the \(d\)-dimensional sphere, where \(T\) is an infinite subset of \(\mathbb N\).  We show that, if \(d\ge 2\), then a finite subset of \(S^d\) has infinite harmonic strength if and only if it is antipodal.
For \(d=1\), we show that infinite strength spherical designs are exactly cyclotomic designs, and we characterize their existence in terms of certain \(0\)-\(1\) polynomials.  We also prove that the harmonic strength of every infinite strength spherical design has the weak GCD property.
Finally, for a given infinite subset \(T\subset \mathbb N\) with the weak GCD property, we give a finite procedure to decide whether there exists \(X\subset S^1\) such that \(\operatorname{Hst}(X)=T\), and apply this criterion to concrete existence and non-existence examples.

\end{abstract}

\maketitle

\section{Introduction}\label{sec:intro}

Spherical designs, introduced by Delsarte, Goethals, and Seidel~\cite{D1977}, are finite sets of points on the unit sphere
\(\Sd^d\subset\RR^{d+1}\) that reproduce exact surface averages of low-degree polynomials.
Concretely, a non-empty finite set \(X\subset \Sd^d\) is a \emph{spherical \(t\)-design} if  

\begin{equation}
   \label{eq:cubature}
  \frac{1}{|\Sd^d|}\int_{\Sd^d}\!f(\xi)\,d\xi
  \;=\;
  \frac{1}{|X|}\sum_{\xi\in X}f(\xi)
  \quad\bigl(\deg f\le t\bigr). 
\end{equation}

When \eqref{eq:cubature} holds for every polynomial of degree at most \(t\) but fails for some polynomial of degree \(t+1\), the integer \(t\) is called the \emph{strength} of \(X\).
As a basic example, for each integer \(t\ge 1\), the vertex set of a regular \((t+1)\)-gon on \(S^1\) forms a spherical \(t\)-design.

A natural generalization replaces the initial segment \(\{1,\dots,t\}\) with an arbitrary set of degrees.
For \(T\subset\mathbb{N}\), we call \(X\subset \Sd^d\) a \emph{spherical \(T\)-design} if
\[
  \sum_{\xi\in X}P(\xi)=0
  \quad
  \bigl(P\in\Harm_{k}(d+1),\;k\in T\bigr),
\]
where \(\Harm_{k}(d+1)\) denotes the space of real homogeneous harmonic polynomials of degree \(k\) in \(d+1\) variables.
The set
\[
  \Hst(X)\coloneqq\left\{k\in\mathbb{N}~\middle|  ~ \sum_{\xi\in X}P(\xi)=0
   \text{ for all } P\in\Harm_{k}(d+1)\right\}
\]
is called the \emph{harmonic strength} of \(X\).
As a trivial example, a regular \((t+1)\)-gon on \(S^1\) is a spherical \( (\mathbb{N}\setminus (t+1)\mathbb{N})\)-design;
in particular,
\( \Hst(X)= \mathbb{N}\setminus (t+1)\mathbb{N}\) in this case.
{As another example, it is well known that for any $X\subset\Sd^d$, if $X$ is antipodal (that is, \(X = -X\)), then $\{t\in\mathbb{N}\mid \GCD(t,2)=1\}\subset \Hst(X)$ holds.}


Systematic investigations of spherical \(T\)-designs and harmonic strength have only begun recently.
Miezaki~\cite{M2013} introduced the notion of spherical \(T\)-designs and, by applying the theory of modular forms to weighted theta series, determined the harmonic strength of the shells of the lattice \(\mathbb{Z}^2\).
{Independently, Bannai et al.~\cite{BOT2015} also introduced the harmonic index and showed a lower bound on the size of $X$ whose harmonic strength includes $t$.}
Hirao, Nozaki, and Tasaka~\cite{HNT2023} studied the shells of the \(D_4\) lattice along similar lines.
They proved that for every positive integer \(m\), the normalized \(2m\)-shell of \(D_4\) is an antipodal spherical \(\{2,4,10\}\)-design on \(S^3\);
in particular, the \(2\)-shell (the \(D_4\) root system) is a tight spherical \(\{2,4,10\}\)-design.
In a subsequent paper, Hirao, Nozaki, and Tasaka~\cite{HNT2025} extended these methods to other spherical designs on \(\Sd^3\), determining their harmonic strength and establishing several uniqueness results.
More recently, Misawa, Munemasa, and Sawa~\cite{MMS24} investigated spherical designs whose harmonic strength consists only of odd degrees.
They determined that, for each positive integer \(m\), the smallest size of a non-antipodal (that is, \(X\neq -X\)) spherical design with \(\Hst(X)=\{1,3,\dots,2m-1\}\) is \(2m+1\), and obtained an analogous optimal result for interval designs.
{Although there exist several studies on harmonic strength, the inverse problem remains open; that is, for a given $d$ and $T\subset\mathbb{N}$, whether there exists an $X\subset\Sd^d$ such that $\Hst(X)=T$.}
For $d=1$, we~\cite{MN25} resolved the finite inverse problem for harmonic strength on \(S^1\).
We first constructed, for each integer \(t\ge1\), a \(5\)-point spherical design \(X\subset \Sd^1\) with \(\Hst(X)=\{t\}\), and proved that \(5\) is the minimal possible cardinality for such a design.
Building on this, we then showed that for every finite set \(T\subset\mathbb{N}\), there exists a spherical \(T\)-design \(X\subset \Sd^1\) with \(\Hst(X)=T\).
Motivated by these developments, we now ask how far such phenomena extend to arbitrary subsets \(T\subset\mathbb{N}\), possibly infinite, and to higher dimensions.
Hereafter, we refer to a finite subset $X\subset\Sd^d$ with an infinite \(\Hst(X)\) as an \emph{infinite strength spherical design}.
The following problems naturally arise.

\begin{Pb}\label{prob:classification}
Classify all infinite strength spherical designs \(X\subset \Sd^d\).
\end{Pb}

The main purpose of this paper is to provide definitive answers to \Cref{prob:classification} when $d\geq2$, and to obtain strong restrictions when $d=1$.
Our first result shows that non-antipodal designs with infinite harmonic strength cannot occur in $d\geq2$.

\begin{thm}\label{thm:antipodal}
  For any integer $d \geq 2$, a finite subset $X\subset \Sd^d$ is an infinite strength spherical design if and only if $X$ is an antipodal set.
\end{thm}

{
Although \Cref{thm:antipodal} was already noted in \cite{BDK2001}, the authors did not provide an explicit upper bound for the harmonic index of a given $X\subset\Sd^d$.
In contrast, we explicitly determine this upper bound using the inner norm.
Additionally, we investigate a lower bound on the size of spherical designs and observe a difference from Fisher's inequality.
} 

When $d=1$, we show that infinite strength spherical designs must be cyclotomic designs, which generalize antipodal sets.
\begin{thm}\label{thm:infmain}
An $X \subset S^1$ is an infinite strength spherical design if and only if $X$ is a cyclotomic design.
\end{thm}

The detailed definition of a cyclotomic design is provided in Section \ref{sec:dim1}.
Additionally, we establish properties of the harmonic strength of infinite strength spherical designs.

\begin{dfn}[Weak GCD property and GCD property]
Let \(T\subset\mathbb N\). 
We say that \(T\) has the weak GCD property if there exist a finite subset \(N\subset\mathbb N\), a positive integer \(\lambda\),
and a subset \(T_S\subset\{d\in\mathbb N\mid d\mid\lambda\}\) such that
\[
  T=N\cup\{j\in\mathbb N\mid \GCD(j,\lambda)\in T_S\}.
\]
If one can take \(N=\emptyset\), then we say that \(T\) has the GCD property.

When \(T\) has the weak GCD property, its period \(\lambda_p\) is defined by
\[
\lambda_p\coloneq\min\left\{
\lambda\in\mathbb N
\,\middle|\,
\begin{array}{l}
\text{there exist a finite set }N\subset\mathbb N
\text{ and }T_S\subset\{d\in\mathbb N\mid d\mid\lambda\}\\
\text{such that }
T=N\cup\{j\in\mathbb N\mid \GCD(j,\lambda)\in T_S\}
\end{array}
\right\}.
\]
\end{dfn}

\begin{thm}\label{thm:hst}
For any infinite strength spherical design $X\subset\Sd^d$, $\Hst(X)$ has the weak GCD property.
In particular, when $d>1$, the period of $\Hst(X)$ is $2$.
\end{thm}

Moreover, we consider the inverse problem of \Cref{thm:hst} when $d=1$ as follows:

\begin{Pb}\label{prob:inverse}
    Given a subset \(T\subset\mathbb{N}\) which is an infinite set, does there exist a spherical design \(X\subset S^{1}\) such that \(\Hst(X)=T\)?
\end{Pb}

We immediately conclude that if \(T\subset\NN\) does not possess the weak GCD property, then there is no finite set \(X\subset S^1\) such that \(\Hst(X)=T\).
For example, there is no finite set \(X\subset S^1\) such that \(\Hst(X)=\{2^k\mid k\in\NN\}.\)
To solve \Cref{prob:inverse}, we provide a method to determine whether there exists a finite set \(X\subset S^1\) such that \(\Hst(X)=T\), even when \(T\) has the weak GCD property.

We show that the existence of an infinite strength spherical design in \(\Sd^1\) is equivalent to the existence of \(0\)-\(1\) polynomials whose joint greatest common divisor with \(x^{\lambda_p}-1\) is prescribed by cyclotomic polynomials.

\begin{thm}\label{thm:decide}
Let \(T\) be an infinite subset of \(\mathbb{N}\) which has the weak GCD
property with period \(\lambda_p\), and suppose that
\[
  T=N\cup\{t\in\NN\mid\GCD(t,\lambda_p)\in T_S\},
\]
where \(N\subset\mathbb N\) is finite and
\[
  T_S\subset\{t\in\mathbb N\mid t\mid\lambda_p\}.
\]
Then an \(X\subset\Sd^1\) with \(\Hst(X)=T\) exists if and only if there exist
a positive integer \(m\) and nonzero polynomials
\[
  f_1,\ldots,f_m\in\mathbb Q[x]
\]
such that the following three conditions hold:
\begin{itemize}
  \item[$\cdot$] \(\deg f_j<\lambda_p\) for all \(j=1,\ldots,m\).
  \item[$\cdot$] All coefficients of each \(f_j\) are \(0\) or \(1\).
  \item[$\cdot$]
  \[
    \GCD(f_1,\ldots,f_m,x^{\lambda_p}-1)
    =
    \prod_{t\in T_S}\Phi_{\lambda_p/t}.
  \]
\end{itemize}
\end{thm}

Note that the number of nonzero \(0\)-\(1\) polynomials of degree strictly less than \(\lambda_p\) is exactly \(2^{\lambda_p}-1\). Therefore, \Cref{thm:decide} implies that the existence of an \(X\) such that \(\Hst(X)=T\) can be decided by checking finitely many subfamilies of this finite set of polynomials.

For example, using \Cref{thm:decide}, we deduce that there does not exist an $X$ such that 
\[
  \Hst(X)=N\cup\{j\in\NN\mid \GCD(j,6)\in\{2,3\}\}
\]
for any finite subset $N$, which is discussed in Proposition \ref{prop:six} in Section \ref{sec:inverse}.
To resolve these problems, we use tools from the theory of linear recurrence sequences and transcendental number theory, namely the Skolem--Mahler--Lech theorem~\cite{Lech,Skolem} and the Lindemann--Weierstrass theorem~\cite{B1975}. The latter is used in its linear-independence form: exponentials of distinct algebraic numbers are linearly independent over \(\overline{\mathbb Q}\).

This paper is organized as follows.
In \Cref{sec:inf}, we discuss infinite strength spherical designs in $\Sd^d$ for $d\geq 2$.
In particular, we provide the proof of \Cref{thm:antipodal} in \Cref{ssec:proveinf} and discuss a lower bound on the number of points in \Cref{ssec:appinf}.
In \Cref{sec:dim1}, we discuss infinite strength spherical designs in $\Sd^1$.
We first introduce the definition of cyclotomic designs in \Cref{ssec:def}, and discuss the harmonic strength of cyclotomic designs in \Cref{ssec:hst}.
After that, we prove \Cref{thm:infmain} 
and \Cref{thm:hst} in \Cref{ssec:inf}. 
In \Cref{sec:inverse}, we prove \Cref{thm:decide}, which provides one of the answers to \Cref{prob:inverse}.
To achieve this, we establish some preliminary theorems in \Cref{ssec:inverseEq} and \Cref{ssec:inverseGCD}, and complete the proof of \Cref{thm:decide} in \Cref{ssec:inversepro}.
Furthermore, we solve certain existence problems using \Cref{thm:decide} in \Cref{ssec:inversepro}.

\section{Infinite strength spherical design on $\Sd^d$, where $d\geq2$}\label{sec:inf}  

In this section, we study infinite strength spherical designs when $d\geq2$; that is, finite subsets $X$ of $\Sd^d$ with an infinite $\Hst(X)$.
In particular, we prove that infinite strength spherical designs are always antipodal sets when $d\geq2$.
First, we define antipodal sets.

\begin{dfn}\label{df:antipodal}
  A finite subset $X\subset\Sd^d$ is called antipodal if $X=-X$ holds; that is, for any $\xi\in X$, $-\xi\in X$ holds.
\end{dfn}

This section is organized as follows.
In \Cref{ssec:preinf}, we define some notation and review previous results regarding Jacobi polynomials, which are closely related to spherical designs.
In \Cref{ssec:proveinf}, we provide the proof of \Cref{thm:antipodal}.
In \Cref{ssec:appinf}, we discuss a lower bound on the size of $X\subset\Sd^d$ with $t\in \Hst(X)$, and compare it with the Fisher-type inequality.


\subsection{Preliminaries for Section \ref{sec:inf}}\label{ssec:preinf}

Let $\mathbb{S}^d$ be the $d$-dimensional sphere, and let $\langle \cdot, \cdot \rangle$ be the standard inner product.
In this section, we assume that $d\geq2$.

It is known that spherical $T$-designs can be defined using Gegenbauer polynomials, which are a specific class of Jacobi polynomials.
Hereafter, we denote the Jacobi polynomial by $P_k^{(\alpha, \beta)}(x)$.

\begin{dfn}[Gegenbauer polynomial]
  The Gegenbauer polynomial of degree $k$, denoted by $Q_k^{(d)}$, is defined as the specific Jacobi polynomial:
\[Q_k^{(d)}(x) = P_k^{(\frac{d-2}{2}, \frac{d-2}{2})}(x).\]
\end{dfn}

\begin{thm}[Lemma 2.1 in \cite{BOT2015}]\label{thm:spd}
Let $T \subset \mathbb{N}$.
A finite set $X \subset \mathbb{S}^d$ is a spherical $T$-design if and only if $\sum_{x,y \in X} Q_k^{(d)}(\langle x, y \rangle) = 0$ holds for all $k \in T$.
\end{thm}

An inequality for Jacobi polynomials was obtained by Haagerup and Schlichtkrull~\cite{HS2014}.
We define the gamma function and $\hat{P}_k^{(\alpha, \beta)}(x)$ as follows:
\begin{align*}
 \Gamma(z) &= \int_{0}^{\infty} t^{z-1} e^{-t} \, dt \quad (\mathrm{Re}(z) > 0), \\
 \hat{P}_k^{(\alpha, \beta)}(x) &= \left( \frac{(2k+\alpha+\beta+1)\Gamma(k+1)\Gamma(k+\alpha+\beta+1)}{2^{\alpha+\beta+1}\Gamma(k+\alpha+1)\Gamma(k+\beta+1)} \right)^{\frac{1}{2}} P_k^{(\alpha, \beta)}(x).
\end{align*}

\begin{thm}[\cite{HS2014}]\label{thm:Jine}
    If $\alpha\geq0$ and $\beta\geq0$, then there exists a constant $0 < C < 12$ such that
\[
(1-x^2)^{\frac{1}{4}}\sqrt{(1-x)^\alpha(1+x)^\beta}|\hat{P}_k^{(\alpha, \beta)}(\langle \xi_1, \xi_2 \rangle)|
\leq \frac{C}{\sqrt{2}}\left(2k+\alpha+\beta+1\right)^{\frac{1}{4}}.
\]
\end{thm}

The main result in the paper~\cite{HS2014} is an inequality concerning weighted Jacobi polynomials, but it is also explicitly noted therein that the inequality for the normalized Jacobi polynomials $\hat{P}_k$ follows naturally as a direct consequence.
Note that due to the condition on $\alpha$ and $\beta$, \Cref{thm:Jine} can be applied to Gegenbauer polynomials only when $d\geq2$.

\subsection{Proof of Theorem \ref{thm:antipodal}}\label{ssec:proveinf}

In this subsection, we provide a proof of Theorem \ref{thm:antipodal}.
First, we present a lemma and a theorem necessary to prove Theorem \ref{thm:antipodal}.
\begin{lem}\label{lem:Gegen_ineq}
Let $d\geq2$, let $X$ be a finite subset of $\mathbb{S}^d$, and suppose that 
\textcolor{black}{there exists at least one pair $(x,y)\in X^2$ such that $|\la x, y\ra|\neq1$.}
Define 
\[
  l_X \coloneq \max \{ |\langle x, y \rangle|
: x, y \in X, |\langle x, y \rangle|\neq1 \}.
\]
For $\xi_1, \xi_2 \in X$, if $|\langle \xi_1, \xi_2 \rangle|
\neq 1$, then
\[
|\hat{P}_k^{(\frac{d-2}{2}, \frac{d-2}{2})}(\langle \xi_1, \xi_2 \rangle)| < 6\sqrt{2}\left(2k+d-1\right)^{\frac{1}{4}}\left(\frac{1}{1-l_X^2}\right)^{\frac{d-1}{4}}.
\]
\end{lem}

\begin{proof}
Substituting $\alpha=\beta=\frac{d-2}{2}$ and $C=12$ into the inequality of Theorem \ref{thm:Jine}, we obtain the following inequality:
\begin{equation*}
|(1-x^2)^{\frac{d-1}{4}}|
|\hat{P}_k^{(\frac{d-2}{2},\frac{d-2}{2})}(x)| < 6\sqrt{2}\left(2k+d-1\right)^{\frac{1}{4}}.
\end{equation*}
Since $|\la\xi_1,\xi_2\ra|\leq l_X<1$, we have
\begin{align*}
|\hat{P}_k^{(\frac{d-2}{2},\frac{d-2}{2})}(\langle \xi_1, \xi_2 \rangle)|
&< 6\sqrt{2}\left(2k+d-1\right)^{\frac{1}{4}}\left(\frac{1}{1-(\langle \xi_1, \xi_2 \rangle)^2}\right)^{\frac{d-1}{4}} \\
&\le 6\sqrt{2}\left(2k+d-1\right)^{\frac{1}{4}}\left(\frac{1}{1-l_X^2}\right)^{\frac{d-1}{4}}.
\end{align*}
Note that since $|\langle \xi_1, \xi_2 \rangle| \neq 1$, $1-l_X^2 \neq 0$.
\end{proof}

\begin{rem}
\textcolor{black}{
Note that we can apply \Cref{lem:Gegen_ineq} for almost all $X\subset\Sd^d$.
In fact, when $|X|\geq3$ or when $|X|=2$ and $X$ is not antipodal, we can use \Cref{lem:Gegen_ineq}.
This is because the case when $|\langle x,y\rangle|=1$ for all $(x,y)\in X^2$ is only two cases: when $|X|=1$ or when $|X|=2$ and  $X$ is an antipodal set.
}
\end{rem}

Next, we show that a sufficiently large $k$ cannot belong to the harmonic strength of $X\subset\Sd^d$ when $d \ge 2$.
\begin{thm}\label{thm:main}
Let $X$ be a finite subset of $\mathbb{S}^d$, where $d \ge 2$, and \textcolor{black}{assume that there exists at least one pair $(x,y)\in X^2$ such that $|\la x, y\ra|\neq1$.}
Define the following constants for $X$:
\begin{align*}
a_X &\coloneq \#\{(\xi_1,\xi_2) \in X^2 \mid \langle \xi_1, \xi_2 \rangle = -1\}, \\
A_X &\coloneq  |X|^2 - (|X| + a_X),\\
l_X &\coloneq \max \{ |\langle x, y \rangle|
: x, y \in X, |\langle x, y \rangle|\neq1 \}, \\
\end{align*}

If $\Hst(X)\cap\{2k\mid k\in\NN\}$ is nonempty, then 
\[
\max(\Hst(X)\cap\{2k\mid k\in\NN\})<\left( \left(\frac{3A_X}{|X|+a_X}\right)^4\cdot\left(1-l_X^2\right)^{-(d-1)}\cdot\left(2^{2d+3}\Gamma\left(\frac{d}{2}\right)^4\right)\right)^{\frac{1}{2d-3}}.
\] 
Additionally, if $|X| - a_X > 0$ and $\Hst(X)\cap\{2k-1\mid k\in\NN\}$ is nonempty, then
\[
  \max(\Hst(X)\cap\{2k-1\mid k\in\NN\})<\left( \left(\frac{3A_X}{|X|-a_X}\right)^4\cdot\left(1-l_X^2\right)^{-(d-1)}\cdot\left(2^{2d+3}\Gamma\left(\frac{d}{2}\right)^4\right)\right)^{\frac{1}{2d-3}}.
\] 
\end{thm}

\begin{proof}
In this proof, we denote $\hat{P}_k^{(\frac{d-2}{2}, \frac{d-2}{2})}(x)$ by $\widehat{Q}_k^{(d)}(x)$.
First, we provide the explicit value of $\widehat{Q}_k^{(d)}(\pm1)$.
By a straightforward calculation, 
\begin{equation}
  \left|\widehat{Q}_k^{(d)}(1)\right|=\left|\widehat{Q}_k^{(d)}(-1)\right|=\left( \frac{(2k + d - 1)\Gamma(k + d - 1)}{2^{d-1} \Gamma(\frac{d}{2})^2 \Gamma(k + 1)} \right)^{\frac{1}{2}}=C_df_d(k)^{\frac{1}{2}},\label{eq:edge}
\end{equation}
where 
\begin{align*}
  C_d &= \frac{1}{2^{\frac{d-1}{2}}\Gamma(\frac{d}{2})},\\
  f_d(k) &= \frac{(2k + d -1)\Gamma(k + d - 1)}{\Gamma(k + 1)}.
\end{align*}
Note that since $k$ and $d$ are integers, $f_d(k)$ is a polynomial in $k\in\mathbb{Z}$ of degree $d-1$ as follows:
\[
f_d(k) =\begin{cases}
  (2k + d-1)\prod_{j=1}^{d-2}(k + j) &d \ge 3\\
  2k + 1 &d=2
\end{cases}.
\]
Additionally, to simplify notation, we define the constant $U_{X,d}$ and the function $g_d(k)$ as follows:
\begin{align*}
  U_{X,d} &\coloneq 6\sqrt{2} \left( \frac{1}{1 - l_X^2} \right)^{\frac{d-1}{4}},\\
  g_d(k)&=2k+d-1.
\end{align*}
Then, Lemma~\ref{lem:Gegen_ineq} can be rephrased as
\begin{equation}
  \left|\widehat{Q}_k^{(d)}(x)\right|\leq U_{X,d}g_d(k)^{\frac{1}{4}}.  \label{eq:lem}
\end{equation}
Then, 
\begin{align}
\sum_{x,y \in X} \widehat{Q}_k^{(d)}(\langle \xi_1,\xi_2 \rangle)
&= |X|
\widehat{Q}_k^{(d)}(1) + a_X \widehat{Q}_k^{(d)}(-1) + \sum_{\substack{\xi_1,\xi_2 \in X \\ \langle \xi_1,\xi_2 \rangle \neq \pm 1}} \widehat{Q}_k^{(d)}(\langle \xi_1,\xi_2 \rangle) \nonumber\\
&\ge |X|
\widehat{Q}_k^{(d)}(1) + a_X \widehat{Q}_k^{(d)}(-1) - \sum_{\substack{\xi_1,\xi_2 \in X \\ \langle \xi_1,\xi_2 \rangle \neq \pm 1}} \left|\widehat{Q}_k^{(d)}(\langle \xi_1,\xi_2 \rangle)\right| \nonumber\\
&> |X|
\widehat{Q}_k^{(d)}(1) + a_X \widehat{Q}_k^{(d)}(-1) - A_XU_{X,d} g_d(k)^{\frac{1}{4}}. \label{eq:gegen}
\end{align}
Note that in the last inequality we used Equation~\ref{eq:lem}, where $A_X$ is the number of pairs $(\xi_1,\xi_2)$ in $X$ for which $|\la\xi_1,\xi_2\ra|\neq1$.
\begin{enumerate}
  \item \textbf{When $k$ is even}
  
  We show that if 
  \[k\geq\left( \left(\frac{3A_X}{|X|+a_X}\right)^4\cdot\left(1-l_X^2\right)^{-(d-1)}\cdot\left(2^{2d+3}\Gamma\left(\frac{d}{2}\right)^4\right)\right)^{\frac{1}{2d-3}}\] 
  then $k\notin\Hst(X)$.
Since $\widehat{Q}^{(d)}_k$ is an even function when $k$ is even, from Equation~\ref{eq:edge} and Equation~\ref{eq:gegen},
  \begin{align}
    \sum_{x,y \in X} \widehat{Q}_k^{(d)}(\langle \xi_1,\xi_2 \rangle)&>(|X|+a_X)|\widehat{Q}_k^{(d)}(1)|
- A_XU_{X,d} g_d(k)^{\frac{1}{4}}\nonumber\\
    &=(|X| + a_X)C_d f_d(k)^{\frac{1}{2}} - A_XU_{X,d} g_d(k)^{\frac{1}{4}} .\label{eq:ine}
  \end{align}
Here, we show that
\begin{equation}
(|X| + a_X)C_d f_d(k)^{\frac{1}{2}} - A_XU_{X,d} g_d(k)^{\frac{1}{4}} \ge 0 \label{eq:ine2}
\end{equation}
when 
\[
  k \geq \left( \frac{A_X^4U_{X,d}^4}{2(|X|+a_X)^4 C_d^4} \right)^{\frac{1}{2d-3}}.
\]
Since \(f_d(k)\ge (2k+d-1)k^{d-2}=g_d(k)k^{d-2}\) holds and \(g_d(k)=2k+d-1>2k\), we have
\begin{align*}
  (|X|+a_X)^4C_{d}^4(f_d(k))^2
  &\ge (|X|+a_X)^4C_{d}^4(g_d(k)k^{d-2})^2 \\
  &= (|X|+a_X)^4C_{d}^4 k^{2d-4}g_d(k)^2 \\
  &> 2(|X|+a_X)^4C_{d}^4 k^{2d-3}g_d(k)\\
  &\geq A_X^4U_{X,d}^4 g_d(k).
\end{align*}
Therefore, the left-hand side of Equation~\ref{eq:ine2} is positive; in particular, Equation~\ref{eq:ine2} holds.
From Equation~\ref{eq:ine} and Equation~\ref{eq:ine2}, if $k$ satisfies these conditions, then 
\[
  \sum_{x,y \in X} \widehat{Q}_k^{(d)}(\langle \xi_1,\xi_2 \rangle) > (|X| + a_X)C_d f_d(k)^{\frac{1}{2}} - \textcolor{black}{A_XU_{X,d}} g_d(k)^{\frac{1}{4}} > 0.
\]
From Theorem~\ref{thm:spd}, the above inequality implies $k \notin \mathrm{Hst}(X)$.
Finally, from a straightforward calculation, we obtain 
\begin{equation*}
  \frac{A_X^4U_{X,d}^4}{2(|X|+a_X)^4 C_d^4}=\left(\frac{3A_X}{|X|+a_X}\right)^4\left(1-l_X^2\right)^{-(d-1)}\left(2^{2d+3}\Gamma\left(\frac{d}{2}\right)^4\right).
\end{equation*}

\item \textbf{When $k$ is odd}

Since $\widehat{Q}^{(d)}_k$ is an odd function when $k$ is odd, from
Equation~\ref{eq:edge} and Equation~\ref{eq:gegen}, we have
\[
  \sum_{x,y \in X} \widehat{Q}_k^{(d)}(\langle x,y \rangle)
  >
  (|X| - a_X)C_d f_d(k)^{\frac{1}{2}}
  -
  A_XU_{X,d} g_d(k)^{\frac{1}{4}}.
\]
By an argument similar to the case when $k$ is even, we can show that
\[
  (|X| - a_X)C_d f_d(k)^{\frac{1}{2}}
  -
  A_XU_{X,d} g_d(k)^{\frac{1}{4}}
  \geq 0
\]
when $k$ and $|X|-a_X$ satisfy the specified conditions. Hence
\[
  \sum_{x,y \in X} \widehat{Q}_k^{(d)}(\langle x,y \rangle)>0.
\]
Therefore, by \Cref{thm:spd}, we have $k\notin\Hst(X)$.
\end{enumerate}

\end{proof}

\begin{rem}
  The condition $|X| - a_X = 0$ is equivalent to $X$ being an antipodal set.
Therefore, Theorem~\ref{thm:main} implies that if $X$ is not antipodal, $d \ge 2$, and $k$ is odd and sufficiently large compared to $d$ and $|X|$, then $k \notin \mathrm{Hst}(X)$.
  Even if $X$ is an antipodal set, from the latter inequality of Theorem~\ref{thm:main}, $\Hst(X)\cap\{2k\mid k\in\NN\}$ is always a finite set.
Furthermore, if $X$ is antipodal, that is, $a_X=|X|$ and $A_X=|X|^2-2|X|$,
  \[
  \max(\Hst(X)\cap\{2k\mid k\in\NN\})<\left( 3^4(|X|-2)^4\cdot\left(1-l_X^2\right)^{-(d-1)}\cdot \left(2^{2d-1}\Gamma\left(\frac{d}{2}\right)^4\right)\right)^{\frac{1}{2d-3}}.
\]
\end{rem}

Finally, we show that a property of $\Hst(X)$ when $X$ is an antipodal set.

\begin{thm}\label{thm:HSTanti} 
  Let $X\subset\Sd^d$ be an antipodal set.
  Then, 
  \[
    \Hst(X)\supset\{2k-1\mid k\in\NN\}.
\]
\end{thm}

While \Cref{thm:HSTanti} is a well-known result with a simple proof, we present it here for the sake of completeness.

\begin{proof}[Proof of \Cref{thm:HSTanti}]
Since $X$ is an antipodal set, we can partition $X$ into disjoint subsets $X_1$ and $X_2$ such that $X_1=-X_2$.
Then, since $\widehat{Q}_{2k-1}^{(d)}$ is an odd function,
  \begin{align*}
    \sum_{x,y \in X} \widehat{Q}_{2k-1}^{(d)}(\langle x,y \rangle)&=\sum_{x,y \in X_1} \widehat{Q}_{2k-1}^{(d)}(\langle x,y \rangle)+\sum_{x,y \in X_2} \widehat{Q}_{2k-1}^{(d)}(\langle x,y \rangle)+\sum_{\substack{x\in X_1\\y \in X_2}} \widehat{Q}_{2k-1}^{(d)}(\langle x,y \rangle)+\sum_{\substack{x\in X_2\\y \in X_1}} \widehat{Q}_{2k-1}^{(d)}(\langle x,y \rangle)\\
    &=\sum_{x,y \in X_1} \widehat{Q}_{2k-1}^{(d)}(\langle x,y \rangle)+\sum_{x,y \in X_1} \widehat{Q}_{2k-1}^{(d)}(\langle -x,-y \rangle)+\sum_{x,y\in X_1} \widehat{Q}_{2k-1}^{(d)}(\langle x,-y \rangle)+\sum_{x,y\in X_1} \widehat{Q}_{2k-1}^{(d)}(\langle -x,y \rangle)\\
    &=2\left(\sum_{x,y \in X_1} \widehat{Q}_{2k-1}^{(d)}(\langle x,y \rangle)-\sum_{x,y\in X_1} \widehat{Q}_{2k-1}^{(d)}(\langle x,y \rangle)\right)\\
    &=0.
\end{align*}
\end{proof}

\begin{thm}\label{thm:hstGEQ2}
  Let $X$ be a finite subset of $\Sd^d$ and assume $d\geq2$.
\textcolor{black}{If $X$ is an antipodal set} \textcolor{black}{and $|X|\geq3$,} then there exists a finite subset $N\subset\{k\in\NN\mid \GCD(k,2)=2\}$ such that 
  \[
    \Hst(X)=N\cup\{k\in\NN\mid \GCD(k,2)=1\}.
  \]
Additionally, if $N\neq\emptyset$ then 
  \[
    \max(N)\leq \left( \left(\frac{3A_X}{|X|+a_X}\right)^4\left(1-l_X^2\right)^{-(d-1)}\left(2^{2d+3}\Gamma\left(\frac{d}{2}\right)^4\right)\right)^{\frac{1}{2d-3}}.
\]
\textcolor{black}{In particular, $X$ is an infinite strength spherical design.}
\end{thm}
\begin{proof}
  We denote 
  \[
  \Hst(X)\cap\{k\in\NN\mid \GCD(k,2)=1\}
  \] 
  and 
  \[
  \Hst(X)\cap\{k\in\NN\mid \GCD(k,2)=2\}
  \] 
  by $\Hst(X)_{\text{odd}}$ and $\Hst(X)_{\text{even}}$, respectively.
From \textcolor{black}{Theorem~\ref{thm:HSTanti}}, 
\[
  \Hst(X)_{\text{odd}}=\{k\in\NN\mid \GCD(k,2)=1\}.
\] 
  From Theorem~\ref{thm:main}, $\Hst(X)_{\text{even}}$ is a finite set and the upper bound of the maximum value is given.
Therefore, 
  \begin{align*}
    \Hst(X)&=\Hst(X)\cap\NN\\
    &=(\Hst(X)\cap\{k\in\NN\mid \GCD(k,2)=2\})\cup(\Hst(X)\cap\{k\in\NN\mid \GCD(k,2)=1\})\\
    &=N\cup \{k\in\NN\mid \GCD(k,2)=1\},
  \end{align*}
  where $N=\Hst(X)_{\text{even}}$ is a finite set.
\end{proof}
\begin{rem}\label{rem:two-point-antipodal}
    \textcolor{black}{When $|X|=2$ and $X$ is an antipodal set, then 
    \[
    \Hst(X)=\{k\in\mathbb{N}\mid \GCD(k,2)=1\}.
    \]}
\end{rem}

Then, we prove Theorem~\ref{thm:antipodal}.

\begin{proof}[Proof of Theorem~\ref{thm:antipodal}]
  Let $d\geq2$ and let $X$ be an infinite strength spherical design in $\Sd^d$; that is, $\Hst(X)$ is an infinite set.
  \textcolor{black}{First, we assume that $|X|\geq3$.}
From Theorem~\ref{thm:main}, $\Hst(X)\cap\{2k\mid k\in\NN\}$ is always a finite set.
Therefore, $\Hst(X)\cap\{2k-1\mid k\in\NN\}$ must be infinite.
However, \Cref{thm:main} also implies if $X$ is not antipodal then $\Hst(X)\cap\{2k-1\mid k\in\NN\}$ must be finite.
Thus, $X$ must be an antipodal set.
Conversely, assume that $X$ is an antipodal set.
Then, from \Cref{thm:hstGEQ2}, $X$ is infinite strength spherical design.
\textcolor{black}{From the above, when $|X|\geq3$, we obtain $X$ is an infinite strength spherical design if and only if $X$ is antipodal. }

\textcolor{black}{
It remains to consider the cases \(|X|\le 2\).
If \(|X|=1\), then, for every \(k\in\NN\),
\[
  \sum_{x,y\in X} Q_k^{(d)}(\langle x,y\rangle)=Q_k^{(d)}(1)\neq0.
\]
Hence, by Theorem~\ref{thm:spd}, we have \(\Hst(X)=\emptyset\).
If \(|X|=2\) and \(X\) is not antipodal, then there exists a pair \((x,y)\in X^2\) such that \(|\langle x,y\rangle|\ne1\), so Theorem~\ref{thm:main} applies and implies that \(\Hst(X)\) is finite.
If \(|X|=2\) and \(X\) is antipodal, then \(X=\{x,-x\}\), and \(\Hst(X)=2\NN-1\).}

\end{proof}

\subsection{Application of Theorem~\ref{thm:main}}\label{ssec:appinf}

In this subsection, for a given $k\in\mathbb{N}$, we consider a lower bound on the size of $X$ for which $k\in\Hst(X)$.
First, we state a well-known result regarding such a lower bound, referred to as a Fisher-type inequality.
\begin{thm}[\cite{D1977}]\label{thm:Fish}
  Let $X$ be any finite subset of $\Sd^{d}$.
If $\{1,\ldots,2k\}\subset\Hst(X)$, then 
  \[
      |X|\geq\begin{pmatrix}
      d+k\\
      d
    \end{pmatrix}+\begin{pmatrix}
      d+k-1\\
      d
    \end{pmatrix},
  \]
  and if $\{1,\ldots,2k-1\}\subset\Hst(X)$, then 
  \[
      |X|\geq2\begin{pmatrix}
      d+k-1\\
      d
    \end{pmatrix}.
\]
\end{thm}

From Theorem~\ref{thm:Fish}, it can be restated that if $\{1,\ldots,k\}\subset\Hst(X)$, then the size of $X$ grows at least on the order of $k^d$ with respect to $k$.
On the other hand, another lower bound was obtained by Bannai et al.~\cite{BOT2015}.
\begin{thm}[\cite{BOT2015}]\label{thm:Fish2}
  Let $2k$ be any even number and let $X$ be any finite subset of $\Sd^{d}$ with $2k\in\Hst(X)$.
Then, there exists a constant $c_{d+1,t}$ such that the following inequality holds:
  \[
    |X|\geq1+\frac{1}{c_{d+1,t}}\left(\begin{pmatrix}
      d+2k\\2k
    \end{pmatrix}-\begin{pmatrix}
      d+2k-2\\2k-2
    \end{pmatrix}\right).
\]
\end{thm}

The constant \(c_{d+1,t}\) in Theorem~\ref{thm:Fish2} is studied in detail in~\cite{BOT2015}.
For example, when $d=1$, $c_{2,t}=2$ for all $t$.
It is also shown in~\cite{BOT2015} that for a fixed $d\geq2$, as $k$ goes to infinity, the limit of the lower bound of Theorem~\ref{thm:Fish2} is constant with respect to $k$ (Proposition~4.1 in~\cite{BOT2015}).
Therefore, Theorem~\ref{thm:Fish2} implies that if $2k\in\Hst(X)$ and $k$ is sufficiently large, then the size of $X$ is at least $b_d$, which is a constant determined solely by $d$.
  

Actually, from Theorem~\ref{thm:main}, we can also deduce a lower bound on the size of $X$.
Hereafter, given $X\subset\Sd^d$, we use the constant $l_X$ as defined in Theorem~\ref{thm:main}.
\begin{thm}\label{thm:lb}
  Let $X\subset\Sd^d$ be an antipodal set with $|X|\geq4$ and the maximal nontrivial inner product $\alpha$; that is, 
  \[\alpha=\max \{ \langle x, y \rangle : x, y \in X, |\langle x, y \rangle|\neq1 \}.\] 
  For any natural number $k$, if $2k\in\Hst(X)$, then 
  \[|X|>2+\frac{\left(1-\alpha^2\right)^{\frac{d-1}{4}}}{3\sqrt2\Gamma\left(\frac{d}{2}\right)}k^{\frac{2d-3}{4}}.\]
\end{thm}

\begin{proof}
  Let $l_X$ be the real number defined in Theorem~\ref{thm:main}.
If $X$ is antipodal, then $\alpha=l_X$.
  Therefore, from Theorem~\ref{thm:main},
  \[\max(\Hst(X)\cap\{2k\mid k\in\NN\})<\left(3^4(|X|-2)^4\left(1-\alpha^2\right)^{-(d-1)}\left(2^{2d-1}\Gamma\left(\frac{d}{2}\right)^4\right)\right)^{\frac{1}{2d-3}}.\]
  When $2k\in \Hst(X)$, then 
  \begin{align*}
     2k&<\left(3^4(|X|-2)^4\left(1-\alpha^2\right)^{-(d-1)}\left(2^{2d-1}\Gamma\left(\frac{d}{2}\right)^4\right)\right)^{\frac{1}{2d-3}}\\
    \Longleftrightarrow k^{2d-3}&<324(|X|-2)^4\left(1-\alpha^2\right)^{-(d-1)}\Gamma\left(\frac{d}{2}\right)^4.
\end{align*}
  Hence, 
  \begin{align*}
    (|X|-2)^4&>\frac{k^{2d-3}\left(1-\alpha^2\right)^{d-1}}{\left(3\sqrt2\Gamma\left(\frac{d}{2}\right)\right)^4}\\
    \Longleftrightarrow |X|&>2+\frac{\left(1-\alpha^2\right)^{\frac{d-1}{4}}}{3\sqrt2\Gamma\left(\frac{d}{2}\right)}k^{\frac{2d-3}{4}}.
  \end{align*}
\end{proof}




Additionally, when $X$ is non-antipodal, we can also obtain a lower bound on the size of $X$.
\begin{thm}\label{thm:lb2}
  Let $X\subset\Sd^d$ be a non-antipodal set, and \textcolor{black}{assume that there exists at least one pair $x,y\in X$ such that $|\langle x, y \rangle|\neq1$}. Define 
\[
  l_X \coloneq \max \{ |\langle x, y \rangle|
: x, y \in X, |\langle x, y \rangle|\neq1 \}.
\]
  For any natural number $k$, if $2k\in\Hst(X)$, then 
  \[
    |X|>1+\frac{(1-l_X^2)^{\frac{d-1}{4}}}{6\sqrt{2}\Gamma(\frac{d}{2})}k^{\frac{2d-3}{4}},
  \]
  and if $2k-1\in\Hst(X)$, then
  \[
    |X|>1+\left(\frac{\left(1-l_X^2\right)^{\frac{d-1}{4}}}{6\sqrt{2}\Gamma(\frac{d}{2})}\right)^{\frac{1}{2}}(k-1)^{\frac{2d-3}{8}}.
\]
\end{thm}

\begin{proof}
  Let $a_X$ be the constant defined in Theorem~\ref{thm:main}.
  Since $X$ is not antipodal, $0\leq a_X<|X|$.
From Theorem~\ref{thm:main},
  \begin{align}
    2k&<\left( \left(\frac{3(|X|^2-(|X|+a_X))}{|X|+a_X}\right)^4\cdot\left(1-l_X^2\right)^{-(d-1)}\cdot\left(2^{2d+3}\Gamma\left(\frac{d}{2}\right)^4\right)\right)^{\frac{1}{2d-3}}, \label{eq:even}\\
    2k-1&<\left( \left(\frac{3(|X|^2-(|X|+a_X))}{|X|-a_X}\right)^4\cdot\left(1-l_X^2\right)^{-(d-1)}\cdot\left(2^{2d+3}\Gamma\left(\frac{d}{2}\right)^4\right)\right)^{\frac{1}{2d-3}}.\label{eq:odd}
  \end{align}
  First, we establish upper bounds for the right-hand sides of Equations~\ref{eq:even} and \ref{eq:odd}.
Since
  \begin{align*}
    \frac{|X|^2-(|X|+a_X)}{|X|-a_X}\leq (|X|-1)^2
  \end{align*}
  and 
  \begin{align*}
    \frac{|X|^2-(|X|+a_X)}{|X|+a_X}\leq |X|-1,
  \end{align*}
  the following upper bounds also hold:
  \begin{align}
    \text{Right hand side of Equation~\ref{eq:even}}&\leq \left(3^4(|X|-1)^4\cdot\left(1-l_X^2\right)^{-(d-1)}\cdot\left(2^{2d+3}\Gamma\left(\frac{d}{2}\right)^4\right)\right)^{\frac{1}{2d-3}}, \label{eq:even2}\\
    \text{Right hand side of Equation~\ref{eq:odd}}&\leq\left(3^4(|X|-1)^8\cdot\left(1-l_X^2\right)^{-(d-1)}\cdot\left(2^{2d+3}\Gamma\left(\frac{d}{2}\right)^4\right)\right)^{\frac{1}{2d-3}}.\label{eq:odd2}
  \end{align}
  From Equation~\ref{eq:even} and Equation~\ref{eq:even2},
  \begin{align*}
    (2k)^{2d-3}&<3^4(|X|-1)^4\cdot\left(1-l_X^2\right)^{-(d-1)}\cdot\left(2^{2d+3}\Gamma\left(\frac{d}{2}\right)^4\right)\\
    \Longleftrightarrow 3^4(|X|-1)^4&>\left(1-l_X^2\right)^{d-1}\frac{(2k)^{2d-3}}{2^{2d+3}\Gamma\left(\frac{d}{2}\right)^4}\\
    \Longleftrightarrow |X|-1&>\frac{\left(1-l_X^2\right)^{\frac{d-1}{4}}}{6\sqrt2\Gamma(\frac{d}{2})}k^{\frac{2d-3}{4}}.
\end{align*}
  From Equation~\ref{eq:odd} and Equation~\ref{eq:odd2},
  \begin{align*}
    (2k-1)^{2d-3}&<3^4(|X|-1)^8\cdot\left(1-l_X^2\right)^{-(d-1)}\cdot\left(2^{2d+3}\Gamma\left(\frac{d}{2}\right)^4\right)\\
    \Longleftrightarrow 3^4(|X|-1)^8&>\left(1-l_X^2\right)^{d-1}\frac{(2k-1)^{2d-3}}{2^{2d+3}\Gamma\left(\frac{d}{2}\right)^4}\\
    &>\left(1-l_X^2\right)^{d-1}\frac{2^{2d-3}(k-1)^{2d-3}}{2^{2d+3}\Gamma\left(\frac{d}{2}\right)^4}\\
    \Longleftrightarrow |X|-1&>\left(\frac{\left(1-l_X^2\right)^{\frac{d-1}{4}}}{6\sqrt2\Gamma(\frac{d}{2})}\right)^{\frac{1}{2}}(k-1)^{\frac{2d-3}{8}}.
\end{align*}
\end{proof}

\begin{rem}
  Theorem~\ref{thm:Fish} implies that if \(\{1,\ldots,2k\}\subset\Hst(X)\), then \(|X|\) is at least of order \(k^{d}\) as \(k\to\infty\) with \(d\) fixed.
  On the other hand, Theorems~\ref{thm:lb} and~\ref{thm:lb2} imply that, under the corresponding geometric assumptions, if \(2k\in\Hst(X)\), then \(|X|\) is at least of order \(k^{(2d-3)/4}\).
  These bounds are complementary to the Fisher-type inequality: they apply to a single harmonic index \(2k\), rather than requiring the whole initial segment \(\{1,\ldots,2k\}\) to be contained in \(\Hst(X)\).

  They should also be compared with Theorem~\ref{thm:Fish2}. Unlike Theorem~\ref{thm:Fish2}, our bounds use additional geometric information about \(X\), such as antipodality and the parameter \(l_X\), or equivalently the maximal nontrivial inner product in the antipodal case.  Thus these estimates do not simply improve Theorem~\ref{thm:Fish2} in full generality; rather, they give complementary lower bounds under additional geometric assumptions.
\end{rem}

\section{Infinite strength spherical design on $\Sd^1$}\label{sec:dim1}
In this section, we classify spherical designs with infinite harmonic strength when $d=1$.
The classification of spherical designs in $\Sd^1$ was studied by Hong~\cite{H1982}.
In the paper~\cite{H1982}, they identified a class of infinite strength spherical designs called ``group-type designs''.
However, the paper did not consider the harmonic strength of these spherical designs.
In contrast, in a previous work, we proved that for any finite set $N\subset\NN$, there exists a spherical design $X\subset\Sd^1$ such that $\Hst(X)=N$~\cite{MN25}.
Therefore, in this section, we focus on cases where the harmonic strength is an infinite set.
Since we cannot apply Theorem~\ref{thm:main} when $d=1$, we need other methods to classify infinite strength spherical designs.
However, Hong~\cite{H1982} showed an equivalence condition for the harmonic strength of $X\subset\Sd^1$ by identifying \(\Sd^1\) with \(\{z\in\mathbb{C}\mid|z|=1\}\).
Using this condition, we approach this problem via algebraic methods.
Specifically, we define cyclotomic designs, which generalize group-type designs and antipodal sets, and show that an infinite strength spherical design in $\Sd^1$ must be a cyclotomic design.
Additionally, we show that cyclotomic designs can be represented by polynomials, and their harmonic strength is primarily determined by the greatest common divisor of these polynomials and certain cyclotomic polynomials.
As a result, we prove \Cref{thm:infmain} and \Cref{thm:hst}.


This section is organized as follows.
In \Cref{ssec:def}, we introduce some terminology and define cyclotomic designs.
Additionally, we present a property of cyclotomic sets for later use.
In \Cref{ssec:hst}, we discuss the harmonic strength of cyclotomic designs.
In \Cref{ssec:inf}, we present the proof of \Cref{thm:infmain}.

Before proceeding, we state a characterization of harmonic strength established by Hong~\cite{H1982}.
Hereafter, we identify \(\Sd^1\) with \(\{z\in\mathbb{C}\mid|z|=1\}=\mathbb{T}\) and we denote the \(k\)-th complex moment of $X$ by $P_k(X)$; that is,
\(P_k(X)\coloneq\sum_{x\in X}x^{k}.\)

\begin{lem}[Lemma 1 of Hong~\cite{H1982}, Lemma 2.1 in~\cite{BOT2015}]\label{lem:Hong}
  For any $X\subset\Sd^1$,
\[
  \Hst(X)=\left\{k\in\mathbb{N}  ~\middle|
~P_k(X)=0\right\}.
\]
\end{lem}

Note that Hong only showed the equivalence condition for $\{1,\ldots,t\}\in\Hst(X)$, but using Lemma 2.1 in~\cite{BOT2015}, it can be reformulated as \Cref{lem:Hong}.
We primarily use Lemma~\ref{lem:Hong} for calculating the harmonic strength.

\subsection{The definition of cyclotomic designs}\label{ssec:def}

For any $X\subset\mathbb{T}$, we denote by $X^\lambda$ the set of $\lambda$-th powers of elements in $X$, that is,
\[
  X^\lambda\coloneqq\{\xi^\lambda\mid \xi\in X\}.
\]
We also denote 
\[
  \mu_\lambda\coloneqq\{z\in\mathbb{T}\mid z^\lambda=1\}
\]
and
\[
  \mu_\infty\coloneqq \bigcup_{\lambda\in\NN}\mu_\lambda
\]
as the group of roots of unity in $\CC$.
We denote $\zeta_\lambda$ as a generator of $\mu_\lambda$.
We define an equivalence relation on $\mathbb{T}$
\[
  \xi_1\sim_{\mu_{\lambda}}\xi_2\Longleftrightarrow \xi_1\xi_2^{-1}\in\mu_{\lambda},
\]
where $\lambda\in\NN\cup\{\infty\}$.
Additionally, we denote
  \[X/\!\sim_{\mu_{\lambda}}=\{X_{j,\lambda}\}_{j=1}^{m_\lambda}\]
as the partition of $X$ induced by the equivalence relation $\sim_{\mu_{\lambda}}$, where $\lambda\in\NN\cup\{\infty\}$.
Specifically, when $\lambda=\infty$, we denote $X/\!\sim_{\mu_{\infty}}=\{X_{j}\}_{j=1}^m$.
\begin{ex}
    Let 
    \[
        X=\{1,-1,i,-i,e^i,e^{(1+\tfrac{2\pi}{3})i},e^{(1+\tfrac{4\pi}{3})i},e^{2i}\}.
\]
    Then,
    \begin{align*}
        X/\!\sim_{\mu_{1}}&=\{\{1\},\{-1\},\{i\},\{-i\},\{e^i\},\{e^{(1+\tfrac{2\pi}{3})i}\},\{e^{(1+\tfrac{4\pi}{3})i}\},\{e^{2i}\}\},\\
        X/\!\sim_{\mu_{2}}&=\{\{1,-1\},\{i,-i\},\{e^i\},\{e^{(1+\tfrac{2\pi}{3})i}\},\{e^{(1+\tfrac{4\pi}{3})i}\},\{e^{2i}\}\},\\
        X/\!\sim_{\mu_{3}}&=\{\{1\},\{-1\},\{i\},\{-i\},\{e^i,e^{(1+\tfrac{2\pi}{3})i},e^{(1+\tfrac{4\pi}{3})i}\},\{e^{2i}\}\},\\
        X/\!\sim_{\mu_{4}}&=\{\{1,-1,i,-i\},\{e^i\},\{e^{(1+\tfrac{2\pi}{3})i}\},\{e^{(1+\tfrac{4\pi}{3})i}\},\{e^{2i}\}\},\\
        X/\!\sim_{\mu_{6}}&=\{\{1,-1\},\{i,-i\},\{e^i,e^{(1+\tfrac{2\pi}{3})i},e^{(1+\tfrac{4\pi}{3})i}\},\{e^{2i}\}\},\\
        X/\!\sim_{\mu_{12}}&=\{\{1,-1,i,-i\},\{e^i,e^{(1+\tfrac{2\pi}{3})i},e^{(1+\tfrac{4\pi}{3})i}\},\{e^{2i}\}\}.
\end{align*}
    Especially, 
    \[
      X/\!\sim_{\mu_{\lambda}}=\begin{cases*}
                    X/\!\sim_{\mu_{1}} & $\GCD(\lambda,12)=1$\\
                    X/\!\sim_{\mu_{2}} & $\GCD(\lambda,12)=2$\\
                    X/\!\sim_{\mu_{3}} & $\GCD(\lambda,12)=3$\\
                    X/\!\sim_{\mu_{4}} & $\GCD(\lambda,12)=4$\\
                    X/\!\sim_{\mu_{6}} & $\GCD(\lambda,12)=6$\\
                    X/\!\sim_{\mu_{12}} & $\GCD(\lambda,12)=12,\lambda=\infty$
        \end{cases*}.
    \]
\end{ex}

\begin{rem}
Let $X\subset\Sd^1$ and let $X^\lambda=\{y_1,\cdots,y_{m_\lambda}\}$.
Since $xy^{-1}\in\mu_\lambda$ if and only if $x^\lambda=y^\lambda$, we can rephrase $X/\!\sim_{\mu_{\lambda}}$ as a partition of $X$ as follows:
  \[
    X_{j,\lambda}\coloneq\{\xi\in X\mid \xi^\lambda=y_j\}.
\]
  Additionally, for some $\lambda\in\NN$, $X/\!\sim_{\mu_{\lambda}}=\{X_{j,\lambda}\}_{j=1}^{m_\lambda}$ is a finer partition than or equal to $X/\!\sim_{\mu_{\infty}}=\{X_{j}\}_{j=1}^{m}$; 
  that is, $m\leq m_\lambda$, and $X_{j}$ is a union of one or more members of $X/\!\sim_{\mu_{\lambda}}$.
This is because if $x\sim_{\mu_\lambda}y$, then $x\sim_{\mu_\infty}y$.
\end{rem}

\begin{dfn}[Polygon subset]
  Let $X$ be a finite subset of $\mathbb{T}$ with $|X|\geq2$.
We say $X$ is a polygon subset if $x\sim_{\mu_{\infty}} y$ holds for any $x,y\in X$.
Additionally, we define the period of a polygon subset $X$ as
  \[
    \lambda=|\langle\{xy^{-1}\mid x,y\in X\}\rangle|.
\]
\end{dfn}

\begin{rem}
\(X\) is a polygon subset whose period divides \(\lambda\) if and only if \(|X/\sim_{\mu_\lambda}|=1\).
\end{rem}

\begin{rem}
  Let $X\subset\Sd^1$ and suppose that $X/\!\sim_{\mu_{\lambda}}=\{X_{j,\lambda}\}_{j=1}^{m_\lambda}$.
  If $|X_{j,\lambda}|\geq2$ for all $j$, then each $X_{j,\lambda}$ is a polygon subset whose period divides $\lambda$.
\end{rem}
We also define the disjoint union of polygon subsets.

\begin{dfn}[Cyclotomic set]
  Let $X$ be a finite subset of $\mathbb{T}$, and let
  \[X/\!\sim_{\mu_{\infty}}=\{X_j\}_{j=1}^m\]
  denote the partition of $X$ induced by the equivalence relation $\sim_{\mu_{\infty}}$.
Then, $X$ is called a cyclotomic set if $|X_j|\geq2$ holds for all $j$.
We define 
  \[
    \lambda=\left|\left\langle\bigcup_{j=1}^m\{xy^{-1}\mid x,y\in X_j\}\right\rangle\right|=\LCM\{\lambda_j\}_{j=1}^m
  \] 
  as the period of the cyclotomic set, where $\lambda_j$ is the period of $X_j$.
\end{dfn}

\begin{ex}\label{ex:cyc}
        Let 
        \begin{align*}
            X&=\{1,-1,i,-i,e^i,e^{(1+\tfrac{2\pi}{3})i},e^{(1+\tfrac{4\pi}{3})i},e^{2i}\}, \\
            Y&=\{1,i,-i,e^i,e^{(1+\tfrac{2\pi}{3})i},e^{(1+\tfrac{4\pi}{3})i}\},\\
            Z&=\{1,i,e^i,e^{(1+\tfrac{2\pi}{3})i},e^{(1+\tfrac{4\pi}{3})i}\}.
\end{align*}
    Then, 
    \begin{align*}
        X/\!\sim_{\mu_{\infty}}&=\{\{1,-1,i,-i\},\{e^i,e^{(1+\tfrac{2\pi}{3})i},e^{(1+\tfrac{4\pi}{3})i}\},\{e^{2i}\}\}=\{X_1,X_2,X_3\}, \\
        Y/\!\sim_{\mu_{\infty}}&=\{\{1,i,-i\},\{e^i,e^{(1+\tfrac{2\pi}{3})i},e^{(1+\tfrac{4\pi}{3})i}\}\}=\{Y_1,Y_2\},\\
        Z/\!\sim_{\mu_{\infty}}&=\{\{1,i\},\{e^i,e^{(1+\tfrac{2\pi}{3})i},e^{(1+\tfrac{4\pi}{3})i}\}\}=\{Z_1,Z_2\}.
\end{align*}
    Since $|X_3|=1$, $X$ is not a cyclotomic set.
On the other hand, since $|Y_1|=3$ and $|Y_2|=3$, $Y$ is a cyclotomic set.
Additionally, $Y_1$ is a polygon subset with period $4$, and $Y_2$ is a polygon subset with period $3$.
Therefore, the period of $Y$ is $12$.
    Similarly, $Z$ is a cyclotomic set with period $12$.
\end{ex}


In a previous work~(\cite{H1982}), Hong defined group-type $t$-designs, which form an example of infinite $T$-designs.
\begin{dfn}[Group-type $t$-design~\cite{H1982}]
  Let $X$ be a cyclotomic set with a partition $X/\!\sim_{\mu_{\infty}}=\{X_j\}_{j=1}^m$, and let $\lambda_j$ be the period of $X_j$.
An $X\subset S^1$ is called a group-type $t$-design if $|X_j|=\lambda_j$ and $\lambda_j\geq t+1$ holds for all $j$.
\end{dfn}

Then, we generalize group-type $t$-designs as follows:
\begin{dfn}[Cyclotomic design]
  Let $X$ be a cyclotomic set with partition $X/\!\sim_{\mu_{\infty}}=\{X_j\}_{j=1}^m$ and period $\lambda$.
Then, $X$ is called a cyclotomic design if there exists a $t \in \{1,\ldots,\lambda-1\}$ such that $t \in \Hst(X_j)$ for all $j$.
We call this $t$ a common index of the cyclotomic design.
\end{dfn}

\begin{ex}
    Let $Y$ and $Z$ be cyclotomic sets with period $12$ as in \Cref{ex:cyc}, and let $Y/\!\sim_{\mu_{\infty}}=\{Y_1,Y_2\}$ and $Z/\!\sim_{\mu_{\infty}}=\{Z_1,Z_2\}$.
Then, since $\Hst(Y_1)=\emptyset$, $Y$ is not a cyclotomic design.
    On the other hand, since $\Hst(Z_1)=2\NN\setminus4\NN$ and $\Hst(Z_2)=\NN\setminus3\NN$, 
    \[
        \{2,10\}\subset\Hst(Z_1)\cap\Hst(Z_2).
\]
    Therefore, $Z$ is a cyclotomic design and its common indices are $\{2,10\}$.
Note that since $Z$ is not a group-type $t$-design, cyclotomic designs generalize group-type designs.
\end{ex}



The following property about cyclotomic sets is directly derived from the definition.
\begin{thm}\label{thm:triC}
  Let $X\subset\Sd^1$ and assume that $X/\!\sim_{\mu_{\lambda_0}}=\{X_{j,\lambda_0}\}_{j=1}^{m_{\lambda_0}}$, \textcolor{black}{where $\lambda_0$ is any natural number.}
  If $|X_{j,\lambda_0}|\geq2$ satisfies this property for all $j$, then $X$ is a cyclotomic set.
\end{thm}

\begin{proof}
  Since $\{X_{j,\lambda_0}\}_{j=1}^{m_{\lambda_0}}$ satisfies this property, for all $x\in X$, there exists at least one element $y\in X\setminus\{x\}$ such that $x\sim_{\mu_{\lambda_0}}y$.
From the definition, if $x\sim_{\mu_{\lambda_0}}y$, then $x\sim_{\mu_{\infty}}y$.
  This implies that every member of $X/\!\sim_{\mu_{\infty}}$ contains at least two elements.
Therefore, $X$ is a cyclotomic set.
\end{proof}

\begin{rem}
  Let $X$ be an antipodal set, that is, $X=-X$ holds.
Then, since $\mu_2=\{-1,1\}$, all members of $X/\!\sim_{\mu_{2}}$ have $2$ elements.
  Therefore, from \Cref{thm:triC}, $X$ is a cyclotomic set.
\end{rem}

\subsection{The harmonic strength of cyclotomic set}\label{ssec:hst}

In this subsection, we study the harmonic strength of cyclotomic sets, which is later used in the proof of \Cref{thm:infmain}. 
Especially, we prove the following theorem.

  \begin{thm}\label{thm:pol}
    Let $X$ be a cyclotomic set with period $\lambda$ and $X/\!\sim_{\mu_{\infty}}=\{X_j\}_{j=1}^m$.
If $t\in\bigcap_{j=1}^m\Hst(X_j)$, then
    \[
      \{t + k\lambda \mid k\in\mathbb{Z}\}\cap\mathbb{N} \subset \Hst(X).
\]
    In particular, a cyclotomic design is always an infinite strength spherical design.
\end{thm}

To prove \Cref{thm:pol}, we introduce a generalization of polygon subsets, termed a polynomial set, whose harmonic strength is determined by the zeros of a corresponding polynomial.
\begin{dfn}[Polynomial set]
  Let 
  \[
    f=\sum_{j=0}^d a_jx^j
  \] 
  be a polynomial with indeterminate $x$ and degree $d$, where $a_j\in\{0,1\}$ for all $j\in\{0,\ldots,d\}$.
  Additionally, let $\xi$ and $\alpha$ be any elements of $\mathbb{T}$.
  Then, a polynomial set of $f$, $\xi$, and rotation angle $\alpha$, denoted by $X(f,\xi,\alpha)$, is a subset of $\Sd^1$ defined as follows:
  \[
    X(f,\xi,\alpha)\coloneqq\{\alpha \xi^j \mid a_j=1\}.
\]
\end{dfn}

\begin{rem}
  Let $X=X(f,\xi,\alpha)$ be a polynomial set.
  Then, $f(1)$ represents the number of indices $j\in\{1,\ldots,m\}$ for which $a_j=1$.
Therefore, $|X|=f(1)$ holds if and only if $\xi^{j_1}\neq \xi^{j_2}$ for any pair $j_1$ and $j_2$ satisfying $0\leq j_1<j_2\leq \mathrm{deg}(f)$ and $a_{j_1}=a_{j_2}=1$.
Additionally, if $|X|=f(1)$, then $P_k(X)=\alpha^kf(\xi^{k})$ holds.
  By \Cref{lem:Hong}, this implies that in this case,  $t\in\Hst(X)$ if and only if $\xi^{t}$ is a root of $f$.
  For example, when $f(x)=x^4+x^3+x^2+x+1$ and $\theta=\frac{2\pi}{5}$, then $\xi^{t}=\zeta_5^t$ is a root of $f$ when $t\in \NN\setminus5\NN$.
Therefore, $\Hst(X(f,\zeta_5,1))=\NN\setminus5\NN$.
\end{rem}
  
It is easily shown that polygon subsets are always polynomial sets.
\begin{lem}\label{lem:polygons}
  Let $X$ be a polygon subset whose period divides $\lambda$.
Then, there exists a polynomial $f$ and $\alpha\in\mathbb{T}$ such that $f$ and $\alpha$ satisfy the following three conditions:
  \begin{itemize}
      \item[$\cdot$] The degree of $f$ is strictly less than $\lambda$ and strictly greater than $0$.
      \item[$\cdot$] All coefficients of $f$ are $0$ or $1$.
      \item[$\cdot$] $X=X(f,\zeta_\lambda,\alpha)$.
  \end{itemize}
\end{lem}
\begin{proof}
    Since $X$ is a polygon subset whose period divides $\lambda$, for any element $\xi\in X$, 
  \[
    \xi^{-1}X\coloneqq\{\xi^{-1}x\mid x\in X\}\subset\langle\zeta_\lambda\rangle<\mathbb{T}.
\]
  Therefore, there exists a subset $T\subset\{0,1,\ldots,\lambda-1\}$ such that
  \[
    \xi^{-1}X=\{\zeta_\lambda^t\mid t\in T\}.
\]
  Then, $X=X(f,\zeta_\lambda,\xi)$ holds, where
  \[
    f=\sum_{j=0}^{\lambda-1}a_jx^j
  \]
  and $\{a_j\}_{j=0}^{\lambda-1}$ is a sequence taking values $1$ if $j\in T$, and $0$ otherwise.
\end{proof}

  Note that such a polynomial is not uniquely determined.
  For example, let $X=\{e^{i},e^{i+\frac{\pi}{2}}\}$, $f_1=1+x$, and $f_2=1+x^3$.
Then, $X=X(f_1, \zeta_4, e^i)=X(f_2, \zeta_4, e^{i+\frac{\pi}{2}})$.
  \textcolor{black}{
  From this property, we can determine the harmonic strength of a polygon subset using the greatest common divisor of the corresponding polynomial $f$ and the cyclotomic polynomial.}
Hereafter, we denote the $k$-th cyclotomic polynomial by $\Phi_k$, and for any polynomial $f\in\mathbb{Q}[x]$, we denote the set of zeros of $f$ by $Z(f)$.
\begin{lem}\label{lem:ps}
  Let $\lambda$ be a natural number, $f$ be a polynomial whose coefficients are all $0$ or $1$ with degree strictly less than $\lambda$ and let $\alpha\in\mathbb{T}$. 
  Additionally, let
\[
    \GCD(f,x^{\lambda}-1)=\prod_{t\in T_\lambda}\Phi_{\tfrac{\lambda}{t}},
\]
where if $\GCD(f,x^{\lambda}-1)=1$, then we define $T_\lambda=\emptyset$ and the right-hand side as $1$.
If $f(1)=|X(f,\zeta_\lambda,\alpha)|$, then
\[
  \Hst(X(f,\zeta_\lambda,\alpha)) = \{ t \in \NN \mid \GCD(t, \lambda) \in T_\lambda\}.
\]
\end{lem}

\begin{proof}
We define  
  \[
    M\coloneq Z(f)\cap Z(x^{\lambda}-1)=Z(\GCD(f,x^{\lambda}-1)).
\]
  Since $P_k(X(f,\zeta_\lambda,\alpha))=\alpha^kf(\zeta_\lambda^k)$, $P_t(X)=0$ if and only if $\zeta_\lambda^t$ is a zero of $f$.
Therefore, from \Cref{lem:Hong},
  \begin{equation}
    \Hst(X)=\{t\in \mathbb{N}\mid \zeta_\lambda^t\in M\}.\label{eq:hst}
  \end{equation}
  It is known that the zeros of $\Phi_{\tfrac{\lambda}{t}}$ are as follows:
  \[
    Z(\Phi_{\tfrac{\lambda}{t}})=\{\zeta_\lambda^k\mid \GCD(k,\lambda)=t\}.
\]
  Therefore, 
  \begin{align*}
    M&=\bigcup_{t\in T_\lambda}Z(\Phi_{\tfrac{\lambda}{t}})\\
    &=\bigcup_{t\in T_\lambda}\{\zeta_\lambda^k\mid \GCD(k,\lambda)=t\}\\
    &=\{\zeta_\lambda^k\mid \GCD(k,\lambda)\in T_\lambda\}.
\end{align*}
  From Equation~\ref{eq:hst}, 
  \[
    \Hst(X) = \{ t \in \NN \mid \GCD(t, \lambda) \in T_\lambda\}.
\]
\end{proof}

\begin{ex}
Let $f=x^4+x^3+x^2+x+1=\Phi_5$.
Assume the polygon subset is $X(f,\zeta_5,\alpha)$.
Then, 
\[
  \GCD(f,x^5-1)=\prod_{t\in T_\lambda}\Phi_{\tfrac{5}{t}},
\]
where $T_\lambda=\{1\}$.
Therefore, from Lemma~\ref{lem:ps},
\[
  \Hst(X(f,\zeta_5,\alpha))=\{t\in\NN\mid \GCD(t,5)=1\}=\NN\setminus5\NN.
\]
On the other hand, consider the polygon subset $X(f,\zeta_6,\alpha)$.
Then, 
\[
  \GCD(f,x^6-1)=1=\prod_{t\in T_\lambda}\Phi_{\tfrac{6}{t}},
\]
where $T_\lambda=\emptyset$, and we obtain 
\[
  \Hst(X(f,\zeta_6,\alpha))=\emptyset.
\]
\end{ex}

\begin{rem}\label{rem:1}
  Let $X$ be a polygon subset with period $\lambda$.
  From \Cref{lem:polygons}, there exist a polynomial $f$ and rotation angle $\alpha\in\mathbb{T}$ such that $X=X(f,\zeta_\lambda,\alpha)$.
  In this case, since $f(1)\neq0$ holds, 
  \[
    \Phi_1=x-1\nmid\GCD(f,x^{\lambda}-1).
\]
  Therefore, for any polygon subset with period $\lambda$, 
  \[
    \{k\lambda\mid k\in\mathbb{N}\}\cap\Hst(X)=\emptyset.
\]
\end{rem}

\begin{rem}\label{rem:periodcollapse}
  From \Cref{lem:ps}, the harmonic strength of a polygon subset $X$ has the GCD property.
Therefore, if $\Hst(X)\neq\emptyset$, then there exists a period of $\Hst(X)$.
Note that the period of $\Hst(X)$ is sometimes different from the period of the polygon subset $X$.
For example, let $X=\{1,e^{\frac{\pi i}{6}},e^{\frac{5\pi i}{6}},-1,e^{\frac{3\pi i}{2}}\}$.
  Then, $X$ is a polygon subset with period $12$, and $X=X(f,\zeta_{12},1)$, where $f=x^9+x^6+x^5+x+1$.
Since 
  \[
    \GCD(f,x^{12}-1)=x^4-x^2+1=\Phi_{12},
  \]
  using \Cref{lem:ps}, the harmonic strength of $X$ is as follows:
  \[
    \Hst(X)=\{t\in\mathbb{N}\mid\GCD(t,12)\in\{1\}\}.
\]
  However, since 
  \[
    \{t\in\mathbb{N}\mid\GCD(t,12)\in\{1\}\}=\{t\in\mathbb{N}\mid\GCD(t,6)\in\{1\}\},
  \]
  the period of $\Hst(X)$ is $6$.
\end{rem}

  Let $X$ be a polygon subset with period $\lambda$.
\Cref{lem:polygons} guarantees the existence of a polynomial $f$ and an $\alpha\in\mathbb{C}$ such that $X=X(f,\zeta_\lambda,\alpha)$.
However, since such a polynomial $f$ is not uniquely determined, it is important to consider the dependency of the choice of $f$ when using \Cref{lem:ps}.
We can show that this result is independent of the choice of $f$.
\begin{prop}\label{prop:indpendency}
    Let $X$ be a polygon subset, and assume that there exist polynomials $f_1$ and $f_2$, a natural number $\lambda$, and rotation angles $\alpha_1$ and $\alpha_2$ such that 
    $X=X(f_1,\zeta_\lambda,\alpha_1)=X(f_2,\zeta_\lambda,\alpha_2)$ and $f_1(1)=f_2(1)=|X|$.
Then, 
    \[
      \GCD(f_1,x^\lambda-1)=\GCD(f_2,x^\lambda-1).
\]
  \end{prop}
  \begin{proof}
  Since $X=X(f_1,\zeta_\lambda,\alpha_1)=X(f_2,\zeta_\lambda,\alpha_2)$, for any point $v\in X$, there exist exponents $t_v$ and $t'_v$ such that 
  \begin{align*}
   f_1(x)&=\sum_{v\in X}x^{t_v},\\
   f_2(x)&=\sum_{v\in X}x^{t'_v}.
\end{align*}  
  For any $v\in X$, we obtain 
  \begin{equation}
    v=\alpha_1\zeta_\lambda^{t_{v}}=\alpha_2\zeta_\lambda^{t'_{v}}.\label{eq:p1}
  \end{equation}
  First, we show that if $\alpha_1=\alpha_2=\alpha$, then $f_1\equiv f_2\pmod{x^\lambda-1}$.
From \Cref{eq:p1}, there exists an integer $k_v$ such that $t_{v}-t'_{v}=k_v\lambda$, and we obtain that 
  \[
    f_2(x)= \sum_{v\in X}x^{t_v+k_v\lambda}.
\]
  This implies $f_1\equiv f_2\pmod{x^\lambda-1}$.
  Next, suppose that $\alpha_1\neq\alpha_2$.
  From \Cref{eq:p1}, we obtain \textcolor{black}{$t_v-t_v'$ are independent of the choice of $v$ and } $\alpha_2=\alpha_1\zeta_\lambda^{t_v-t'_v}$.
Therefore, 
  \[
    X=X(x^{|t_v-t_v'|\lambda+t_v-t_v'}f_2(x),\zeta_\lambda,\alpha_1).
  \]
  This implies that $f_1(x)\equiv x^{|t_v-t'_v|\lambda+t_v-t'_v}f_2(x)\pmod{x^\lambda-1}$.
In either case, since $\GCD(x^k,x^\lambda-1)=1$ for all $k\in\mathbb{N}$, $\GCD(f_1,x^{\lambda}-1)=\GCD(f_2,x^{\lambda}-1)$ holds.
\end{proof}

Then, we can establish properties regarding the harmonic strength of cyclotomic sets.

\begin{lem}\label{lem:pol0}
Let $X\subset\Sd^1$ be a finite set and suppose
\[
X/\!\sim_{\mu_\lambda}=\{X_{j,\lambda}\}_{j=1}^{m_\lambda},
\]
where $\lambda\in\mathbb{N}$.
For each $j=1,\ldots,m_\lambda$, let $(f_j,\zeta_\lambda,\alpha_j)$ be a triplet consisting of a polynomial $f_j$ whose coefficients are all $0$ or $1$ with degree strictly less than $\lambda$,
a root of unity $\zeta_\lambda$, and a rotation angle $\alpha_j$ such that
\[
X_{j,\lambda}=X(f_j,\zeta_\lambda,\alpha_j),
\]
and suppose that
\[
    \GCD(f_1,\ldots, f_{m_\lambda}, x^{\lambda}-1)
    =
    \prod_{t\in T_\lambda}\Phi_{\tfrac{\lambda}{t}},
\]
where if
\[
    \GCD(f_1,\ldots, f_{m_\lambda}, x^{\lambda}-1)=1,
\]
then we define $T_{\lambda}=\emptyset$ and the right-hand side as $1$.
Then,
\[
\bigcap_{j=1}^{m_\lambda}\Hst(X_{j,\lambda})
=
\{t\in\NN\mid\GCD(t,\lambda)\in T_\lambda\}.
\]
\end{lem}

\begin{proof}
First, we assume that there exists some $j_0\in\{1,\ldots,m_{\lambda}\}$ such that
$|X_{j_0,\lambda}|=1$.
Then $f_{j_0}$ is a monomial, and hence
\[
\GCD(f_{j_0},x^\lambda-1)=1.
\]
Therefore,
\[
\GCD(f_1,\ldots, f_{m_\lambda}, x^{\lambda}-1)=1.
\]
Thus $T_\lambda=\emptyset$. Additionally, one point cannot be a spherical $t$-design for any
$t\in\NN$. This implies
\[
\Hst(X_{j_0,\lambda})=\emptyset.
\]
Therefore, in this case,
\[
    \bigcap_{j=1}^{m_\lambda}\Hst(X_{j,\lambda})
    =
    \{t\in\NN\mid\GCD(t,\lambda)\in T_\lambda\}
    =
    \emptyset.
\]

Next, we assume that $|X_{j,\lambda}|\geq2$ holds for all $j=1,\ldots,m_\lambda$.
Let
\[
    \GCD(f_j,x^{\lambda}-1)
    =
    \prod_{t\in T_{\lambda,j}}\Phi_{\tfrac{\lambda}{t}}
    \qquad (j=1,\ldots,m_\lambda).
\]
Then, since the degree of $f_j$ is strictly less than $\lambda$, we have
\[
f_j(1)=|X_{j,\lambda}|.
\]
Therefore, from \Cref{lem:ps},
\[
    \Hst(X_{j,\lambda})
    =
    \{t\in\mathbb{N}\mid \GCD(t,\lambda)\in T_{\lambda,j}\}.
\]
By the definition of $T_\lambda$, we have
\[
T_\lambda=\bigcap_{j=1}^{m_\lambda}T_{\lambda,j}.
\]
Therefore,
\[
    \bigcap_{j=1}^{m_\lambda}\Hst(X_{j,\lambda})
    =
    \{t\in\mathbb{N}\mid \GCD(t,\lambda)\in T_{\lambda}\}.
\]
\end{proof}

\begin{rem}
  Since $f_j(1)\neq0$ for all $j$, $\lambda$ is always not in $T_\lambda$.
\end{rem}

\begin{rem}
  In \Cref{lem:pol0}, although the set of polynomials $\{f_j\}_{j=1}^m$ is not uniquely determined, \Cref{prop:indpendency} ensures that $T_{\lambda,j}$ and $T_\lambda$ are uniquely determined.
\end{rem}

\begin{ex}\label{ex:1}
  Let $X=\{1,e^{\frac{2\pi i}{3}},e^{\frac{4\pi i}{3}},e^i,-e^{i}\}$.
  Then, $X$ is a cyclotomic set with period $6$, and $X/\!\sim_{\mu_{\infty}}=\{\{1,e^{\frac{2\pi i}{3}},e^{\frac{4\pi i}{3}}\},\{e^i,-e^{i}\}\}$.
Let $f_1=x^4+x^2+1$ and $f_2=x^3+1$.
  Then, 
  \[X/\!\sim_{\mu_{\infty}}=\{X(f_1,\zeta_6,1),X(f_2,\zeta_2,e^i)\}=\{X_1,X_2\}.\]
  Since 
  \[
    \GCD(f_1,x^{6}-1)=\prod_{j\in\{2,6\}}\Phi_{\frac{6}{j}}
  \]
  and 
  \[
    \GCD(f_2,x^{6}-1)=\prod_{j\in\{3,6\}}\Phi_{\frac{6}{j}},
  \]
  using \Cref{lem:pol0}, the harmonic strength of $\Hst(X_1)\cap\Hst(X_2)$ is as follows:
  \[
    \Hst(X_1)\cap\Hst(X_2)=\{t\in\mathbb{N}\mid\GCD(t,6)\in\{1\}\}.
\]
\end{ex}

Then, we prove \Cref{thm:pol} using \Cref{lem:pol0}.

  \begin{proof}[Proof of \Cref{thm:pol}]
    Since the period of $X$ is $\lambda$, $X/\!\sim_{\mu_\infty}=X/\!\sim_{\mu_\lambda}$.
    From \Cref{lem:polygons}, for each $j\in\{1,\ldots,m\}$, there exists a polynomial $f_j$ and rotation angle $\alpha$ such that $X_j=X(f_j,\zeta_\lambda,\alpha_j)$.
    Therefore, from \Cref{lem:pol0},
    \[
\bigcap_{j=1}^{m_\lambda}\Hst(X_{j,\lambda})=\{t\in\NN\mid\GCD(t,\lambda)\in T_\lambda\},
    \]
    where $T_\lambda$ is defined by the greatest common devisor of $f_1,\ldots,f_m$ and $x^{\lambda}-1$.
    Therefore, for any  $t\in\bigcap_{j=1}^{m_\lambda}\Hst(X_{j,\lambda})$, $ \GCD(t,\lambda)\in T_\lambda$.
    Since $\GCD(t,\lambda)=\GCD(t+k\lambda,\lambda)$, 
    \[
      \{t + k\lambda \mid k\in\mathbb{Z}\}\cap\mathbb{N} \subset \bigcap_{j=1}^m\Hst(X_j).
\]
    Additionally, since $P_k(X)=\sum_{j=1}^mP_k(X_j)$, 
    \[
\bigcap_{j=1}^m\Hst(X_j)\subset\Hst(X).
\]
  \end{proof}

\subsection{The classification of infinite strength spherical designs in $\Sd^1$}\label{ssec:inf}

In this subsection, we present the proof of \Cref{thm:infmain} and \Cref{thm:hst}.
To establish \Cref{thm:infmain}, we prove the following theorem.

  \begin{thm}\label{thm:sufInf}
    If the harmonic strength of $X\subset\Sd^1$ is an infinite set, then $X$ is a cyclotomic design.
\end{thm}

To prove \Cref{thm:sufInf}, we first prepare several lemmas. 
The first lemma establishes a property of the harmonic strength of infinite strength spherical designs.

\begin{lem}\label{lem:lech}
  Let $X$ be a finite subset of $S^1$.
If $\Hst(X)$ is an infinite set, then there exists a finite subset $N\subset\NN$ and integers $t_1,\ldots,t_l,{\lambda_0}\in\NN$ such that 
  \[
    \Hst(X)=N\cup\bigcup_{j=1}^l\{t_j+k{\lambda_0} \mid k\in\mathbb{Z}_{\geq0}\}.
\]
\end{lem}

Lemma~\ref{lem:lech} is immediately proved using the Skolem-Mahler-Lech theorem.

\begin{thm}[\cite{Lech,Skolem}]\label{thm:SkolemLech}
Let $\{c_v\}_{v\in\NN}$ be a complex sequence that satisfies a linear recurrence relation for $v \geq n$ as follows:
\[
  c_v = \alpha_1 c_{v-1} + \alpha_2 c_{v-2} + \cdots + \alpha_n c_{v-n},
\]
and define
\[
  Z(\{c_v\}_{v\in\NN}) \coloneq \{v \in \NN \mid c_v = 0\}.
\]
If $Z(\{c_v\}_{v\in\NN})$ is an infinite set, 
then there exist a finite subset $N\subset\NN$, a positive integer $\lambda_0$, and natural numbers $t_1,\ldots,t_l$ such that 
  \[
    Z(\{c_v\}_{v\in\NN})=N\cup\bigcup_{j=1}^l\{t_j+k{\lambda_0} \mid k\in\mathbb{Z}_{\geq0}\}
  \]
\end{thm}

\begin{proof}[Proof of Lemma~\ref{lem:lech}]
  Let $X=\{\xi_1,\ldots,\xi_n\}$ be a subset of $\Sd^1$, and assume $\Hst(X)$ is an infinite set.
Define
  \[
    E_t(X)\coloneq\sum_{1\leq l_1<l_2<\cdots<l_t\leq n}\prod_{1\leq k\leq t}\xi_{l_k}.
\]
  Then, from Newton's identities, for any $t>n$,
  \[
    P_t(X)=\sum_{j=1}^{n}(-1)^{j-1}E_{j}(X)P_{t-j}(X).
\]
  By \Cref{lem:Hong}, since 
  \[
    \Hst(X)=\{t\in\NN\mid P_t(X)=0\},
  \]
  from Theorem~\ref{thm:SkolemLech}, there exists a finite subset $N\subset\NN$ and integers $t_1,\ldots,t_l,{\lambda_0}\in\NN$ such that 
  \[
    \Hst(X)=N\cup\bigcup_{j=1}^l\{t_j+k{\lambda_0} \mid k\in\mathbb{Z}_{\geq0}\}.
\]
\end{proof}

Lemma~\ref{lem:lech} claims that if $X$ is an infinite strength spherical design, then the harmonic strength of $X$ contains an arithmetic progression.
We can also show that if the harmonic strength of $X$ possesses such a property, then $X$ must be a cyclotomic design.

\begin{lem}\label{lem:hst}

\textcolor{black}{Let \(X\subset S^1\) be finite. Suppose that there exist \(\lambda_0\in\mathbb N\), a finite set \(N\subset\mathbb N\), and a finite set \(T\subset\mathbb N\) such that
\[
  \Hst(X)= N \cup \{t+k\lambda_0\mid t\in T,\ k\in\mathbb Z_{\ge0}\}.
\]
Assume that
\[
X/\!\sim_{\mu_{\lambda_0}}=\{X_{j,\lambda_0}\}_{j=1}^{m_{\lambda_0}}.
\]
Then
\[
  \bigcap_{j=1}^{m_{\lambda_0}}\Hst(X_{j,\lambda_0})= \{t+k\lambda_0\mid t\in T,\ k\in\mathbb Z\}\cap\mathbb{N}.
\]}
\end{lem}

\begin{proof}
Let
\[
  X/\!\sim_{\mu_{\lambda_0}}
  =
  \{X_{j,\lambda_0}\}_{j=1}^{m_{\lambda_0}}.
\]
Put \(m:=m_{\lambda_0}\), and define
\[
  y_j:=\xi_j^{\lambda_0}
  \qquad
  (\xi_j\in X_{j,\lambda_0},\ 1\le j\le m).
\]
Since any two elements of \(X_{j,\lambda_0}\) are equivalent under
\(\sim_{\mu_{\lambda_0}}\), the value \(y_j\) does not depend on the choice of
\(\xi_j\in X_{j,\lambda_0}\). Moreover, the numbers \(y_1,\ldots,y_m\) are
pairwise distinct, because different \(\sim_{\mu_{\lambda_0}}\)-classes have
different \(\lambda_0\)-th powers.
We prove this lemma by contradiction.
  Let $t$ be any element of $T$, and assume that there exists an $l$ such that $t\notin\Hst(X_{l,{\lambda_0}})$.
By direct calculation,
  \begin{align*}
    P_{t+k{\lambda_0}}(X)&=\sum_{j=1}^{m_{\lambda_0}}\sum_{\xi\in X_{j,{\lambda_0}}} \xi^{t+k{\lambda_0}}\\
    &=\sum_{j=1}^{m_{\lambda_0}}\sum_{\xi\in X_{j,{\lambda_0}}}  \xi^{t}(\xi^{{\lambda_0}})^k\\
    &=\sum_{j=1}^{m_{\lambda_0}} y_j^k\sum_{\xi\in X_{j,{\lambda_0}}}\xi^{t}\\
    &=\sum_{j=1}^{m_{\lambda_0}} P_{t}(X_{j,{\lambda_0}})y_j^k.
\end{align*}
  Next, we consider the hyperplane $H$ of $\CC^m$ defined as follows:
  \[
    H=\{(x_1,\cdots,x_{m})\mid P_{t}(X_{1,{\lambda_0}})x_1+P_{t}(X_{2,{\lambda_0}})x_2+\cdots+P_{t}(X_{m,{\lambda_0}})x_{m}=0\}.
\]
  Since $t\notin\Hst(X_{l,{\lambda_0}})$, $P_{t}(X_{l,{\lambda_0}})\neq0$ and $H$ is an $(m-1)$-dimensional space over $\CC^m$.
Furthermore, since $t+k{\lambda_0}\in\Hst(X)$ for all $k\in\mathbb{Z}_{\geq0}$,
  $P_{t+k{\lambda_0}}(X)=0$ and 
  \[
    (y_1^k,\cdots,y_{m}^k)\in H.
  \]
  Here, we define the $m\times m$ matrix $V$ as follows:
  \[
    V\coloneq\begin{pmatrix}
      1&1&\cdots&1\\
      y_1&y_2&\cdots&y_{m}\\
      &\vdots&&\\
      y_1^{m-1}&y_2^{m-1}&\cdots&y^{m-1}_{m}
    \end{pmatrix}.
\]
  Since $(y_1^k,\cdots,y_{m}^{k})\in H$ and $H$ is an $(m-1)$-dimensional space, $\Im(V)\subset H$ implies $\det(V)$ must be $0$.
Using the Vandermonde determinant,
  \[
    \det(V)=\prod_{1\leq j<k\leq m}(y_j-y_k)
  \]
  and $\det(V)\neq0$.
From the above, $\Im(V)\nsubseteq H$, which is a contradiction.
  Therefore, for any $t\in T$ and $j$, $t\in\Hst(X_{j,{\lambda_0}})$ holds.
  \textcolor{black}{From \Cref{lem:ps}, this also implies}
\begin{equation}
    \{t + k{\lambda_0} \mid t\in T, k\in\mathbb{Z}\}\cap\mathbb{N}\subset\bigcap_{j=1}^{m_{\lambda_0}}\Hst(X_{j,{\lambda_0}}).\label{eq:hst31}
\end{equation}  
On the other hand, since $P_k(X)=\sum_{j=1}^{m_{\lambda_0}}P_k(X_{j,\lambda_0})$, we have
  \begin{equation}
\bigcap_{j=1}^{m_{\lambda_0}}\Hst(X_{j,\lambda_0})\subset\Hst(X)=N\cup\{t + k{\lambda_0} \mid t\in T, k\in\mathbb{Z}_{\geq0}\}.\label{eq:hst41}
  \end{equation}
  Furthermore, by \Cref{lem:pol0}, 
\begin{equation}
\bigcap_{j=1}^{m_{\lambda_0}}\Hst(X_{j,\lambda_0})=\{t\in\mathbb{N}\mid \GCD(t,\lambda_0)\in T_{\lambda_0}\},\label{eq:hst311}
\end{equation}
where $T_{\lambda_0}$ is determined by the greatest common divisor of polynomials corresponding to $X_{j,\lambda_0}$.
  Therefore, from \Cref{eq:hst31}, \Cref{eq:hst41} and \Cref{eq:hst311}, we obtain
  \begin{equation*}
      \{t + k{\lambda_0} \mid t\in T, k\in\mathbb{Z}\}\cap\mathbb{N}\subset\{t\in\mathbb{N}\mid \GCD(t,\lambda_0)\in T_{\lambda_0}\}\subset N\cup\{t + k{\lambda_0} \mid t\in T, k\in\mathbb{Z}_{\geq0}\}.
  \end{equation*}
  \textcolor{black}{
This implies
\begin{equation}
    \{t\in\mathbb{N} \mid \gcd(t,\lambda_0)\in T_{\lambda_0}\}\setminus(\{t + k{\lambda_0} \mid t\in T, k\in\mathbb{Z}\}\cap\mathbb{N})\subset N.
\end{equation}
On the other hand,  since the set $\{t\in\mathbb{N} \mid \gcd(t,\lambda_0)\in T_{\lambda_0}\}$ is also periodic, \textcolor{black}{the set
\[
    \{t\in\mathbb{N} \mid \gcd(t,\lambda_0)\in T_{\lambda_0}\}\setminus (\{t + k{\lambda_0} \mid t\in T, k\in\mathbb{Z}\}\cap\mathbb{N})
\]
is also periodic.
Therefore,} if 
\[
    \{t\in\mathbb{N} \mid \gcd(t,\lambda_0)\in T_{\lambda_0}\}\setminus (\{t + k{\lambda_0} \mid t\in T, k\in\mathbb{Z}\}\cap\mathbb{N})\neq\emptyset,
\]
then this difference must be infinite.
Since $N$ is finite, this difference must be empty, and we conclude that
\begin{equation*}
    \{t + k{\lambda_0} \mid t\in T, k\in\mathbb{Z}\}\cap\mathbb{N}=\{t\in\mathbb{N} \mid \gcd(t,\lambda_0)\in T_{\lambda_0}\}.
\end{equation*}}
\end{proof}

Lemma~\ref{lem:hst} only refers to the property of harmonic strengths of $X/\!\sim_{\mu_{\lambda_0}}$, not to that of $X/\!\sim_{\mu_{\infty}}$.
However, this property of $X/\!\sim_{\mu_{\lambda_0}}$ implies that $X$ is a cyclotomic design.
\begin{lem}\label{lem:trivial}
  Let $X$ be a subset of $\Sd^1$, let $X/\!\sim_{\mu_{\lambda_0}}=\{X_{j,\lambda_0}\}_{j=1}^{m_{\lambda_0}}$, and let $X/\!\sim_{\mu_{\infty}}=\{X_{j}\}_{j=1}^{m}$.
  Then,
  \[ \bigcap_{j=1}^{m_{\lambda_0}}\Hst(X_{j,\lambda_0})\subset\bigcap_{j=1}^{m}\Hst(X_{j}).
\]
  In particular, if $\bigcap_{j=1}^{m_{\lambda_0}}\Hst(X_{j,\lambda_0})\neq\emptyset$, then $X$ is a cyclotomic design.
\end{lem}

\begin{proof}
  From the definition of \(X/\!\sim_{\mu_{\lambda_0}}\) and
\(X/\!\sim_{\mu_\infty}\), for any \(j\in\{1,\ldots,m\}\), there exist indices
\(j_1,\ldots,j_\ell\in\{1,\ldots,m_{\lambda_0}\}\) such that
\[
  X_j=\bigcup_{s=1}^{\ell}X_{j_s,\lambda_0}.
\]
Then, for every \(r\in\mathbb N\), we have
\[
  P_r(X_j)=\sum_{s=1}^{\ell}P_r(X_{j_s,\lambda_0}).
\]
Hence
\[
\bigcap_{s=1}^{\ell}\Hst(X_{j_s,\lambda_0})\subset\Hst(X_j).
\]
  Therefore,
  \begin{equation}
\bigcap_{j=1}^{m_{\lambda_0}}\Hst(X_{j,\lambda_0})\subset\bigcap_{j=1}^{m}\Hst(X_{j}).\label{eq:inclusion}
  \end{equation}

\textcolor{black}{
Finally, we show that if
\[
\bigcap_{j=1}^{m_{\lambda_0}}\Hst(X_{j,\lambda_0})\neq\emptyset,
\]
then \(X\) is a cyclotomic design. Assume that this intersection is nonempty.
By our assumption, we must have \(|X_{j,\lambda_0}|\ge2\) for all \(j\), because
a one-point set has empty harmonic strength. Therefore, by \Cref{thm:triC},
\(X\) is a cyclotomic set.
Let \(\lambda\) be the period of this cyclotomic set. Then
\[
  X/\!\sim_{\mu_\infty}=X/\!\sim_{\mu_\lambda}=\{X_j\}_{j=1}^{m}.
\]
Choose
\[
t\in\bigcap_{j=1}^{m_{\lambda_0}}\Hst(X_{j,\lambda_0}).
\]
By \Cref{eq:inclusion}, we have
\[
  t\in\bigcap_{j=1}^{m}\Hst(X_j).
\]
Since each \(X_j\) is a polygon subset whose period divides \(\lambda\),
\(\Hst(X_j)\) is \(\lambda\)-periodic by \Cref{lem:ps}. Moreover, by
\Cref{rem:1}, no positive multiple of \(\lambda\) belongs to \(\Hst(X_j)\).
Hence we may replace \(t\) by its representative
\(t'\in\{1,\ldots,\lambda-1\}\), and still have
\[
  t'\in\bigcap_{j=1}^{m}\Hst(X_j).
\]
Consequently, \(X\) is a cyclotomic design.
}
\end{proof}

\begin{ex}
  Let 
  \[
  X=\{1,-1,i,-i,e^i,e^{(1+\frac{2\pi}{3})i},e^{(1+\frac{4\pi}{3})i}\}
  \]
  be a cyclotomic design with period $12$.
Then, 
  \begin{align*}
    X/\!\sim_{\mu_6}&=\{\{1,-1\},\{i,-i\},\{e^i,e^{(1+\frac{2\pi}{3})i},e^{(1+\frac{4\pi}{3})i}\}\}=\{X_{1,6},X_{2,6},X_{3,6}\}\\
    X/\!\sim_{\mu_{\infty}}&=X/\!\sim_{\mu_{12}}=\{\{1,-1,i,-i\},\{e^i,e^{(1+\frac{2\pi}{3})i},e^{(1+\frac{4\pi}{3})i}\}\}=\{X_{1},X_{2}\}.
\end{align*}
  and
  \begin{align*}
    \Hst(X_{1,6})=\Hst(X_{2,6})&=\{t\in\mathbb{N}\mid \GCD(t,2)=1\},\\
    \Hst(X_{3,6})=\Hst(X_{2})&=\{t\in\mathbb{N}\mid \GCD(t,3)=1\},\\
    \Hst(X_{1})&=\{t\in\mathbb{N}\mid \GCD(t,4)\in\{1,2\}\}.
\end{align*}
  Especially,
  \begin{align*}
    \Hst(X_{1,6})\cap\Hst(X_{2,6})\cap\Hst(X_{3,6})&=\{t\in\mathbb{N}\mid \GCD(t,6)=1\},\\
    \Hst(X_{1})\cap\Hst(X_2)&=\{t\in\mathbb{N}\mid \GCD(t,12)\in\{1,2\}\}.
\end{align*}
  Certainly, 
  \[
  \Hst(X_{1,6})\cap\Hst(X_{2,6})\cap\Hst(X_{3,6})\subset\Hst(X_{1})\cap\Hst(X_2).
  \]
\end{ex}

Using Lemma~\ref{lem:lech}, Lemma~\ref{lem:hst}, and Lemma~\ref{lem:trivial}, we can prove Theorem~\ref{thm:sufInf}.

\begin{proof}[Proof of Theorem~\ref{thm:sufInf}]
  Let $X$ be a finite subset of $\Sd^1$ with an infinite $\Hst(X)$.
From \Cref{lem:lech}, there exists a finite subset $N\subset\NN$ and natural numbers $t_1,\ldots,t_\ell,\lambda\in\NN$ such that
  \[
    \Hst(X)=N\cup\bigcup_{j=1}^\ell\{t_j+k\lambda\mid k\in\mathbb{Z}_{\geq0}\}.
\]
  \textcolor{black}{From \Cref{lem:hst}, $\bigcap_{j=1}^m\Hst(X_{j,\lambda})\neq\emptyset$.}
From \Cref{lem:trivial}, $X$ is a cyclotomic design.
\end{proof}

Then, we can prove Theorem~\ref{thm:infmain}.

\begin{proof}[Proof of Theorem~\ref{thm:infmain}]
  If $X$ is a cyclotomic design, then from Theorem~\ref{thm:pol}, $X$ is an infinite strength spherical design.
  Conversely, if the harmonic strength of $X\subset\Sd^1$ is infinite, then from Theorem~\ref{thm:sufInf}, $X$ is a cyclotomic design.
 
\end{proof}

\textcolor{black}{Finally, we characterize the infinite part of the harmonic strength of cyclotomic designs. 
Let $X$ be a cyclotomic set with period $\lambda$, and assume $X/\!\sim_{\mu_\infty} = X/\!\sim_{\mu_\lambda} = \{X_{j}\}_{j=1}^m$. 
Based on \Cref{lem:pol0}, although we have already characterized the intersection $\bigcap_{j=1}^m\Hst(X_j)$, we have not yet considered whether the set $\Hst(X)\setminus(\bigcap_{j=1}^m\Hst(X_j))$ is finite or not. 
The following theorem shows that this set must be finite.}
  
\begin{thm}\label{thm:gcd}
Let $X$ be a cyclotomic set with period $\lambda$, and suppose $X/\!\sim_{\mu_\infty}=\{X_j\}_{j=1}^m$.
\textcolor{black}{Let $(f_j,\zeta_\lambda,\alpha_j)$ be a triplet consisting of a polynomial $f_j$ whose coefficients are all $0$ or $1$ with degree strictly less than $\lambda$,
 a root of unity $\zeta_\lambda$, and a rotation angle $\alpha_j$ such that
\[
    X_j=X(f_j,\zeta_\lambda,\alpha_j),
\]}
and suppose that
\[
    \textcolor{black}{\GCD(f_1,\ldots, f_m, x^{\lambda}-1)=\prod_{t\in T_\lambda}\Phi_{\tfrac{\lambda}{t}},}
\]
where if $\textcolor{black}{\GCD(f_1, \ldots, f_m, x^{\lambda}-1)=1}$, then we define $T_{\lambda}=\emptyset$ and the right-hand side as $1$.
Then, there exists a finite subset $N\subset\NN$ such that 
\[
    \Hst(X)=N\cup \{t\in\NN\mid\GCD(t,\lambda)\in T_\lambda\}.
\]
\end{thm}

\begin{proof}
  From \Cref{lem:pol0}, 
  \begin{equation}
    \bigcap_{j=1}^m\Hst(X_{j})=\{t\in\mathbb{N}\mid \GCD(t,\lambda)\in T_{\lambda}\}.\label{eq:hst2}
  \end{equation}
  \textcolor{black}{
  If $T_\lambda=\emptyset$, then $X$ is not cyclotomic design. 
  Hence, from \Cref{thm:infmain}, $\Hst(X)$ is finite, \textcolor{black}{that is}, there exists a finite subset $N\subset\mathbb{N}$ such that \[
    \Hst(X)=N\cup\{t\in\mathbb{N}\mid \GCD(t,\lambda)\in T_{\lambda}\}.
  \]
  Assume that $T_\lambda\neq\emptyset$.
  Then, $X$ is a cyclotomic design.
  Therefore, from \Cref{thm:infmain}, $X$ is infinite strength spherical design, and from \Cref{lem:hst}, there exists a finite subset $N\subset\NN$ and a period $\lambda_0\in\NN$ such that 
  \[
    \Hst(X)=N\cup\bigcap_{j=1}^{m_{\lambda_0}}\Hst(X_{j,\lambda_0}),
\] 
where $X/\!\sim_{\mu_{\lambda_0}}=\{X_{j,\lambda_0}\}_{j=1}^{m_{\lambda_0}}$. }
  From \Cref{lem:trivial} and \Cref{eq:hst2}, 
\begin{equation}
\bigcap_{j=1}^{m_{\lambda_0}}\Hst(X_{j,\lambda_0})\subset \bigcap_{j=1}^{m}\Hst(X_{j})=\{t\in\mathbb{N}\mid \GCD(t,\lambda)\in T_{\lambda}\}.\label{eq:hst3}
\end{equation}  
  Since $P_k(X_j)=\sum_{k=1}^{\ell}P_k(X_{j_k,\lambda_0})$,
  \begin{equation}
    \bigcap_{j=1}^{m}\Hst(X_{j})\subset\Hst(X)=N\cup\bigcap_{j=1}^{m_{\lambda_0}}\Hst(X_{j,\lambda_0}).\label{eq:hst4}
  \end{equation}
  Therefore, from \Cref{eq:hst3} and \Cref{eq:hst4},
  \[
    \bigcap_{j=1}^{m_{\lambda_0}}\Hst(X_{j,\lambda_0})\subset \{t\in\mathbb{N}\mid \GCD(t,\lambda)\in T_{\lambda}\}\subset N\cup\bigcap_{j=1}^{m_{\lambda_0}}\Hst(X_{j,\lambda_0}).
\]
\textcolor{black}{
Set
\[
D:=
\{t\in\mathbb{N}\mid \GCD(t,\lambda)\in T_{\lambda}\}
\setminus
\bigcap_{j=1}^{m_{\lambda_0}}\Hst(X_{j,\lambda_0}).
\]
Then \(D\subset N\).
Moreover, \(D\) is periodic with period
\[
  L:=\LCM(\lambda,\lambda_0).
\]
If \(D\neq\emptyset\), then \(D\) is infinite, which contradicts the finiteness
of \(N\). Hence \(D=\emptyset\).
Therefore,
}
\[
  \bigcap_{j=1}^{m_{\lambda_0}}\Hst(X_{j,\lambda_0})=\{t\in\mathbb{N}\mid \GCD(t,\lambda)\in T_{\lambda}\}.
\]
\end{proof}

Then, we can prove \Cref{thm:hst}.

\begin{proof}[Proof of \Cref{thm:hst}]
Let \(X\) be an infinite strength spherical design in \(\Sd^d\).

First, assume that \(d\ge2\). By \Cref{thm:antipodal}, \(X\) is an antipodal
set. If \(|X|=2\), then by \Cref{rem:two-point-antipodal},
\[
  \Hst(X)=\{t\in\NN\mid \GCD(t,2)=1\}.
\]
If \(|X|\ge3\), then by \Cref{thm:hstGEQ2}, there exists a finite subset
\[
  N\subset\{t\in\NN\mid \GCD(t,2)=2\}
\]
such that
\[
  \Hst(X)=N\cup\{t\in\NN\mid \GCD(t,2)=1\}.
\]
Thus, in both cases, \(\Hst(X)\) has the weak GCD property, and its period is
\(2\).

Next, assume that \(d=1\). By \Cref{thm:infmain}, \(X\) is a cyclotomic design.
Therefore, by \Cref{thm:gcd}, there exist a finite subset \(N\subset\NN\), a
positive integer \(\lambda\), and a subset
\[
  T_\lambda\subset\{t\in\mathbb N\mid t\mid\lambda\}
\]
such that
\[
  \Hst(X)=N\cup\{t\in\NN\mid\GCD(t,\lambda)\in T_\lambda\}.
\]
Hence \(\Hst(X)\) has the weak GCD property.
\end{proof}

\section{The existence problem for infinite strength spherical design on $\Sd^1$}\label{sec:inverse}

In this section, we primarily focus on \Cref{prob:inverse}.
\Cref{thm:gcd} imposes a restriction on the harmonic strength of infinite strength spherical designs in $\Sd^1$.
First, we define the GCD property, which characterizes the harmonic strength of a cyclotomic design.


From \Cref{thm:gcd}, the harmonic strength of a cyclotomic design always has the weak GCD property.
This provides a partial answer to \Cref{prob:inverse}; that is, if $T$ does not possess the weak GCD property, then an infinite strength spherical design whose harmonic strength is $T$ does not exist.
For example, there is no finite set \(X\subset S^1\) such that \(\Hst(X)=\{2^k\mid k\in\mathbb Z_{\ge0}\}.\)
Therefore, \Cref{prob:inverse} can be reduced to the following problem:

\begin{Pb}\label{prob:inverse2}
  Let $T$ be an infinite subset of $\mathbb{N}$ which has the weak GCD property with period $\lambda_p$.
Then, does there exist an $X\subset\Sd^1$ such that $\Hst(X)=T$?
\end{Pb}

Then, in this section, we primarily focus on \Cref{prob:inverse2}.
We show that \Cref{prob:inverse2} can be solved by examining the properties of finitely many polynomials.

\begin{thm}[Restatement of \Cref{thm:decide}] \label{thm:redecide}
Let \(T\) be an infinite subset of \(\mathbb{N}\) which has the weak GCD
property with period \(\lambda_p\), and suppose that
\[
  T=N\cup\{t\in\NN\mid\GCD(t,\lambda_p)\in T_S\},
\]
where \(N\subset\mathbb N\) is finite and
\[
  T_S\subset\{t\in\mathbb N\mid t\mid\lambda_p\}.
\]
Then an \(X\subset\Sd^1\) with \(\Hst(X)=T\) exists if and only if there exist
a positive integer \(m\) and nonzero polynomials
\[
  f_1,\ldots,f_m\in\mathbb Q[x]
\]
such that the following three conditions hold:
\begin{itemize}
  \item[$\cdot$] \(\deg f_j<\lambda_p\) for all \(j=1,\ldots,m\).
  \item[$\cdot$] All coefficients of each \(f_j\) are \(0\) or \(1\).
  \item[$\cdot$]
  \[
    \GCD(f_1,\ldots,f_m,x^{\lambda_p}-1)
    =
    \prod_{t\in T_S}\Phi_{\lambda_p/t}.
  \]
\end{itemize}
\end{thm}

Note that the number of nonzero \(0\)-\(1\) polynomials of degree strictly less than \(\lambda_p\) is \(2^{\lambda_p}-1\). Hence one may decide the existence of an \(X\) with \(\Hst(X)=T\) by checking at most \(2^{2^{\lambda_p}-1}\) subfamilies.
In this section, we aim to prove \Cref{thm:decide}.

This section is organized as follows.
First, in \Cref{ssec:inverseEq}, we show that \Cref{prob:inverse2} can be resolved by restricting our attention to the case where $T$ has the GCD property.
Next, in \Cref{ssec:inverseGCD}, we prove a restricted version of \Cref{thm:decide}.
Finally, in \Cref{ssec:inversepro}, we complete the proof of \Cref{thm:decide} and demonstrate how it can be applied to solve certain existence problems.

\subsection{The relationship between existence when $T$ has the GCD property and existence when $T$ has the weak GCD property}\label{ssec:inverseEq}

\begin{thm}\label{thm:eqInverse}
  Let $T_{\lambda_p}$ be an infinite subset of $\mathbb{N}$ which has the GCD property with period $\lambda_p$, and let $N$ be any finite subset of $\mathbb{N}$.
  Then, there exists an $X\subset\Sd^1$ such that $\Hst(X)=T_{\lambda_p}$ if and only if there exists an $X'$ such that $\Hst(X')=N\cup T_{\lambda_p}$.
\end{thm}

To prove \Cref{thm:eqInverse}, first, we establish its sufficiency with the following lemma.
\begin{lem}\label{thm:N}
  Let $T_{\lambda_p}$ be an infinite subset of $\mathbb{N}$ which has the GCD property with period $\lambda_p$,
  and let $X$ be a finite subset of $\mathbb{S}^1$ with $\Hst(X)=T_{\lambda_p}$.
Then, for any finite subset $N\subset\NN$, there exists an $X_N\subset \mathbb{S}^1$ such that $\Hst(X_N)=N\cup T_{\lambda_p}$.
\end{lem}

To prove \Cref{thm:N}, we introduce the following lemmas.
Hereafter, for any subsets $X_1, X_2 \subset S^1$, their product $X_1 \cdot X_2$ is defined as follows:
\[
X_1 \cdot X_2 = \{cd \mid c \in X_1,\, d \in X_2\}.
\]

\begin{lem}[\cite{MN25}]\label{lem:prod}
If $X_1,X_2$ have harmonic strengths $T_1,T_2$ respectively, and $|X_1\cdot X_2|=|X_1||X_2|$, then $X_1\cdot X_2$ has harmonic strength $T_1\cup T_2$.
\end{lem}

\begin{lem}[\cite{MN25}]\label{lem:singleton2}
  Let $D$ be any open interval satisfying
    \[\emptyset\neq D\subseteq\Bigl(-1,\tfrac12\Bigr)\setminus\{-\tfrac14\}.\]
  For any $x\in D$, define
\[
    X^{\frac{1}{t}}(x)\coloneqq 
\left\{1,e^{{i\arccos(x)}/{t}},e^{-{i\arccos(x)}/{t}},
e^{i\arccos\left(-x-\tfrac12\right)/t},
e^{-i\arccos\left(-x-\tfrac12\right)/t}
\right\},
\]
where $t$ is any natural number.
Then, there exist uncountably many $x\in D$ such that $\Hst(X^{\frac{1}{t}}(x))=\{t\}$.
\end{lem}

\begin{proof}[Proof of \Cref{thm:N}]
It is enough to consider the finite set
\[
  N^\circ\coloneq N\setminus T_{\lambda_p},
\]
because \(N\cup T_{\lambda_p}=N^\circ\cup T_{\lambda_p}\).
We prove by induction on \(|N^\circ|\) that there exists a finite set
\(X_{N^\circ}\subset\Sd^1\) such that
\[
  \Hst(X_{N^\circ})=N^\circ\cup T_{\lambda_p}.
\]

If \(N^\circ=\emptyset\), we take \(X_{N^\circ}=X\). Assume that the assertion
has been proved for a finite set \(M\subset N^\circ\), and choose
\(t\in N^\circ\setminus M\). Let \(Y\subset\Sd^1\) satisfy
\[
  \Hst(Y)=M\cup T_{\lambda_p}.
\]
We shall construct \(Y'\) with
\[
  \Hst(Y')=(M\cup\{t\})\cup T_{\lambda_p}.
\]

For \(x\in(-1,\tfrac12)\setminus\{-\tfrac14\}\), put
\[
X(x)\coloneqq
\left\{
1,e^{i\arccos(x)},e^{-i\arccos(x)},
e^{i\arccos(-x-\tfrac12)},e^{-i\arccos(-x-\tfrac12)}
\right\}.
\]
Write
\[
  X(x)=\{g_1(x),\ldots,g_5(x)\}
\]
and
\[
  X^{1/t}(x)=\{h_1(x),\ldots,h_5(x)\},
  \qquad h_i(x)^t=g_i(x).
\]

Define the bad set
\[
D_Y\coloneq
\left\{
x\in\left(-1,\tfrac12\right)\setminus\left\{-\tfrac14\right\}
\,\middle|\,
\left(X^{1/t}(x)\bigl(X^{1/t}(x)\bigr)^{-1}\setminus\{1\}\right)
\cap
\left(YY^{-1}\setminus\{1\}\right)
\neq\emptyset
\right\}.
\]
We show that \(D_Y\) is finite. If \(x\in D_Y\), then there exist
\(i\neq j\) and \(\eta\in YY^{-1}\setminus\{1\}\) such that
\[
  h_i(x)h_j(x)^{-1}=\eta.
\]
Raising both sides to the \(t\)-th power gives
\[
  g_i(x)g_j(x)^{-1}=\eta^t.
\]
Therefore
\[
D_Y\subset
\bigcup_{\substack{1\le i,j\le5\\ i\ne j}}
\ \bigcup_{\eta\in YY^{-1}\setminus\{1\}}
\left\{
x\in\left(-1,\tfrac12\right)\setminus\left\{-\tfrac14\right\}
\,\middle|\,
g_i(x)g_j(x)^{-1}-\eta^t=0
\right\}.
\]
For \(i\ne j\), the function \(g_i(x)g_j(x)^{-1}\) is a nonconstant algebraic
function on \((-1,\tfrac12)\setminus\{-\tfrac14\}\). Hence each zero set in the
right-hand side is finite. Since the union is finite, \(D_Y\) is finite.

Choose a nonempty open interval
\[
  D'\subset
  \left(-1,\tfrac12\right)\setminus\left(\left\{-\tfrac14\right\}\cup D_Y\right).
\]
By \Cref{lem:singleton2}, there exists \(x\in D'\) such that
\[
  \Hst(X^{1/t}(x))=\{t\}.
\]
Since \(x\notin D_Y\), we have
\[
  |Y\cdot X^{1/t}(x)|=|Y|\,|X^{1/t}(x)|.
\]
Thus, by \Cref{lem:prod},
\[
  \Hst(Y\cdot X^{1/t}(x))
  =
  \Hst(Y)\cup\Hst(X^{1/t}(x))
  =
  (M\cup T_{\lambda_p})\cup\{t\}.
\]
Set
\[
  Y'=Y\cdot X^{1/t}(x).
\]
This completes the induction. Therefore, for any finite subset \(N\subset\NN\),
there exists \(X_N\subset\Sd^1\) such that
\[
  \Hst(X_N)=N\cup T_{\lambda_p}.
\]
\end{proof}
Next, we prove the necessity of \Cref{thm:eqInverse} with the following lemma.
\begin{lem}\label{lem:useW}
    Let $X\subset\Sd^1$ and $\Hst(X)=N\cup T_{\lambda_p}$, where $N$ is a finite set and $T_{\lambda_p}$ is an infinite set which has the GCD property with period $\lambda_p$.
Then, there exists a cyclotomic design $X'$ such that $\Hst(X')=T_{\lambda_p}$ and the period of $X'$ divides $\lambda_p$.
\end{lem}

\Cref{lem:useW} can be proved using the following theorem.

\begin{thm}[\cite{B1975}, Theorem 1.4]\label{thm:W}
    Let $\{\alpha_{j}\}_{j=1}^m$ be any distinct algebraic numbers and let $\{\beta_{j}\}_{j=1}^m$ be any algebraic numbers.
If
    \[
      \sum_{j=1}^m\beta_je^{\alpha_j}=0
    \]
    holds, then $\beta_j=0$ for all $j$.
\end{thm}

\begin{proof}[Proof of \Cref{lem:useW}]
From \Cref{lem:hst}, 
    \begin{equation}
        T_{\lambda_p}=\bigcap_{j=1}^{m_{\lambda_p}}\Hst(X_{j,{\lambda_p}}), \label{eq:1}
    \end{equation}
    where $X/\!\sim_{\mu_{{\lambda_p}}}=\{X_{j,{\lambda_p}}\}_{j=1}^{m_{\lambda_p}}$.
Recall that $X_{j,\lambda_p}$ is a polygon subset whose period divides $\lambda_p$ for each $j\in\{1,\ldots,m_{\lambda_p}\}$. 
Therefore, from \Cref{lem:polygons}, there exist a polynomial $f_j$ and a rotation angle $\alpha_j\in\mathbb{T}$ such that $X_{j,{\lambda_p}}=X(f_j,\zeta_{{\lambda_p}},\alpha_j)$.
From \Cref{eq:1}, for all $t\in T_{\lambda_p}$ and $j\in\{1,\ldots,m_{\lambda_p}\}$, 
    \begin{equation}
        f_j(\zeta_{{\lambda_p}}^t)=0\label{eq:2}
    \end{equation}
    and for all $t'\notin T_{\lambda_p}$, there exist some $j$ such that 
    \begin{equation}
        f_j(\zeta_{{\lambda_p}}^{t'})\neq0.\label{eq:3}
    \end{equation}

    Let \(X'_j=X(f_j,\zeta_{\lambda_p},e^{ji})\).
Then any point of \(X'_j\) is of the form
\[
  \zeta_{\lambda_p}^{a}e^{ji}
\]
for some integer \(a\).

We claim that \(X'_{j_1}\cap X'_{j_2}=\emptyset\) whenever \(j_1\neq j_2\).
Suppose, to the contrary, that there exists
\(x\in X'_{j_1}\cap X'_{j_2}\) with \(j_1\neq j_2\).
Then there exist integers \(a_1,a_2\) such that
\[
  x=\zeta_{\lambda_p}^{a_1}e^{j_1i}
   =\zeta_{\lambda_p}^{a_2}e^{j_2i}.
\]
Hence
\[
  e^{(j_1-j_2)i}=\zeta_{\lambda_p}^{a_2-a_1}.
\]
The right-hand side is algebraic. On the other hand, since
\((j_1-j_2)i\) is a nonzero algebraic number, \(e^{(j_1-j_2)i}\)
is transcendental by \Cref{thm:W}. This is a contradiction.
Therefore \(X'_{j_1}\cap X'_{j_2}=\emptyset\).

Define
\[
  X'\coloneqq\bigsqcup_{j=1}^{m_{\lambda_p}}X'_j.
\]
From the definition, 
    \begin{align*}
        P_k(X')
        =\sum_{j=1}^{m_{\lambda_p}}P_k(X'_j)
        =\sum_{j=1}^{m_{\lambda_p}} f_j(\zeta_{\lambda_p}^k)e^{jki}.
\end{align*}
    Then, from \Cref{thm:W}, 
    \textcolor{black}{since 
    \[
    f_j(\zeta_{\lambda_p}^k)\in\overline{\mathbb{Q}}
    \]
    and the values $jki$ are distinct algebraic numbers for distinct $j$, }
    if $P_k(X')=0$, then $f_j(\zeta_{{\lambda_p}}^k)=0$ for all $j$.
From \Cref{eq:2} and \Cref{eq:3}, $P_k(X')=0$ if and only if $k\in T_{\lambda_p}$.
    Therefore, $\Hst(X')=T_{\lambda_p}$.
Additionally, since $X_j'^{{\lambda_p}}=\{e^{j{\lambda_p} i}\}$, the period of $X'$ divides ${\lambda_p}$.
\end{proof}

Using these lemmas, we can show \Cref{thm:eqInverse}.

\begin{proof}[Proof of \Cref{thm:eqInverse}]
  Assume that there exists an $X$ such that $\Hst(X)=T_{\lambda_p}$.
  Then, by \Cref{thm:N}, there exists an $X'$ such that $\Hst(X')=N\cup T_{\lambda_p}$.
  Conversely, assume that there exists an $X$ such that $\Hst(X)=N\cup T_{\lambda_p}$.
  Then, by \Cref{lem:useW}, there also exists an $X'$ such that $\Hst(X')=T_{\lambda_p}$.
\end{proof}



\subsection{The existence problem when $T$ has the GCD property}\label{ssec:inverseGCD}

In this subsection, we prove \Cref{thm:decide} restricted to when $T$ has the GCD property.
\begin{thm}\label{thm:decideGCD}
Let \(T_{\lambda_p}\) be an infinite subset of \(\mathbb N\) with the GCD
property whose period is \(\lambda_p\), and suppose that
\[
  T_{\lambda_p}
  =
  \{t\in\NN\mid\GCD(t,\lambda_p)\in T_S\},
\]
where
\[
  T_S\subset\{t\in\mathbb N\mid t\mid\lambda_p\}.
\]
Then an \(X\subset\Sd^1\) with \(\Hst(X)=T_{\lambda_p}\) exists if and only if
there exist a positive integer \(m\) and nonzero polynomials
\[
  f_1,\ldots,f_m\in\mathbb Q[x]
\]
such that the following three conditions hold:
\begin{itemize}
  \item[$\cdot$] \(\deg f_j<\lambda_p\) for all \(j=1,\ldots,m\).
  \item[$\cdot$] All coefficients of each \(f_j\) are \(0\) or \(1\).
  \item[$\cdot$]
  \[
    \GCD(f_1,\ldots,f_m,x^{\lambda_p}-1)
    =
    \prod_{t\in T_S}\Phi_{\lambda_p/t}.
  \]
\end{itemize}
\end{thm}

\begin{proof}
  Suppose that there exists an $X$ such that $\Hst(X)=T_{\lambda_p}$.
Then, from \Cref{lem:hst}, 
  \begin{equation}
    \bigcap_{j=1}^{m_{{\lambda_p}}}\Hst(X_{j,{\lambda_p}})=T_{\lambda_p}, \label{eq:hstpol}
  \end{equation}
  where $X/\!\sim_{\mu_{\lambda_p}}=\{X_{j,{\lambda_p}}\}_{j=1}^{m_{\lambda_p}}$.
  Since \(X_{j,\lambda_p}\) is a polygon subset whose period divides \(\lambda_p\), from \Cref{lem:polygons} there exist a polynomial \(f_j\) and an \(\alpha_j\in\mathbb T\) satisfying the following three conditions:
  \begin{itemize}
      \item[$\cdot$] The degree of $f_j$ is strictly less than ${\lambda_p}$.
      \item[$\cdot$] All coefficients of $f_j$ are $0$ or $1$.
      \item[$\cdot$] $X_{j,{\lambda_p}}=X(f_j,\zeta_{\lambda_p},\alpha_j)$.
\end{itemize}
  Now, suppose that 
  \[
    \GCD(f_j,x^{{\lambda_p}}-1)=\prod_{t\in T_j}\Phi_{\tfrac{{\lambda_p}}{t}}.
\]
  Then, from \Cref{lem:ps}, 
  \[
    \Hst(X_{j,{\lambda_p}})=\{t\in\mathbb N\mid \gcd(t,\lambda_p)\in T_j\},
\]
  Therefore, $\bigcap_{j=1}^{m_{\lambda_p}}T_j=T_S$, and the set $\{f_j\}_{j=1}^{m_{\lambda_p}}$ satisfies all conditions.

 Conversely, assume that there exist \(f_1,\ldots,f_m\) satisfying all conditions.
For each \(j=1,\ldots,m\), define
\[
  X_j:=X(f_j,\zeta_{\lambda_p},e^{ji}).
\]
By the same disjointness argument as in the proof of \Cref{lem:useW}, the sets \(X_1,\ldots,X_m\) are pairwise disjoint. Put
\[
  X:=\bigsqcup_{j=1}^{m}X_j.
\]

Let
\[
  \GCD(f_j,x^{\lambda_p}-1)
  =
  \prod_{t\in T_j}\Phi_{\lambda_p/t}.
\]
Then, by \Cref{lem:ps},
\[
  \Hst(X_j)=\{t\in\mathbb N\mid \GCD(t,\lambda_p)\in T_j\}.
\]

We claim that
\[
  \Hst(X)=\bigcap_{j=1}^{m}\Hst(X_j).
\]
Indeed, if \(r\in\bigcap_{j=1}^{m}\Hst(X_j)\), then
\[
  P_r(X)=\sum_{j=1}^{m}P_r(X_j)=0,
\]
so \(r\in\Hst(X)\). Conversely, suppose that \(P_r(X)=0\). Since
\[
  P_r(X)=\sum_{j=1}^{m} f_j(\zeta_{\lambda_p}^{r})e^{jri},
\]
and since \(f_j(\zeta_{\lambda_p}^{r})\in\overline{\mathbb Q}\) while the numbers \(jri\) are distinct algebraic numbers as \(j\) varies, \Cref{thm:W} implies
\[
  f_j(\zeta_{\lambda_p}^{r})=0
  \qquad (j=1,\ldots,m).
\]
Thus \(P_r(X_j)=0\) for all \(j\), and hence
\(r\in\bigcap_{j=1}^{m}\Hst(X_j)\). Therefore,
\[
  \Hst(X)=\bigcap_{j=1}^{m}\Hst(X_j)=T_{\lambda_p}.
\]
  and this demonstrates the existence of an $X$ with $\Hst(X)=T_{\lambda_p}$.
\end{proof}

\begin{ex}
The polynomials $\{f_j\}_{j=1}^{m_{\lambda_p}}$ correspond to $X/\!\sim_{\mu_{\lambda_p}}=\{X_{j,\lambda_p}\}_{j=1}^{m_{\lambda_p}}$.
For example, recall the setup in \Cref{rem:periodcollapse}; that is, let 
  \[X=\{1,e^{\frac{\pi i}{6}},e^{\frac{5\pi i}{6}},-1,e^{\frac{3\pi i}{2}}\}.\]
  Now, $\Hst(X)$ has the GCD property with period $6$.
We observe that $X/\!\sim_{\mu_{6}}=\{\{e^{\frac{\pi i}{6}},e^{\frac{5\pi i}{6}},e^{\frac{3\pi i}{2}}\},\{1,-1\}\}$.
  Then,
  \[\{e^{\frac{\pi i}{6}},e^{\frac{5\pi i}{6}},e^{\frac{3\pi i}{2}}\}=X(x^4+x^2+1,\zeta_6,e^{\frac{\pi i}{6}})\,\text{ and }\,\{1,-1\}=X(x^3+1,\zeta_6,1).\]
  Therefore, 
  \[X=X(x^4+x^2+1,\zeta_6,e^{\frac{\pi i}{6}})\sqcup X(x^3+1,\zeta_6,1).\] 
  In this case, the polynomials $\{x^4+x^2+1,x^3+1\}$ satisfy the conditions of \Cref{thm:decideGCD}.
In fact, $\GCD(x^4+x^2+1,x^3+1,x^6-1)=\Phi_6$.
\end{ex}

\subsection{Proof of \Cref{thm:decide} and some applications of \Cref{thm:decide}}\label{ssec:inversepro}

Then, we can prove \Cref{thm:decide} using \Cref{thm:eqInverse} and \Cref{thm:decideGCD}.
\begin{proof}[Proof of \Cref{thm:decide}]
  Since $T$ has the weak GCD property with period $\lambda_p$, there exist a finite subset $N$ and a set $T_{\lambda_p}$ having the GCD property with period $\lambda_p$ such that
    \[
        T=N\cup T_{\lambda_p}.
\]
  From \Cref{thm:eqInverse}, the existence of an $X$ with $\Hst(X)=N\cup T_{\lambda_p}$ is equivalent to the existence of an $X'$ with $\Hst(X')=T_{\lambda_p}$.
Since $T_{\lambda_p}$ has the GCD property with period $\lambda_p$, from \Cref{thm:decideGCD}, the existence of such an $X'$ is equivalent to the existence of the polynomials $\{f_j\}_{j=1}^{m_{\lambda_p}}$.
\end{proof}

We can apply Theorem~\ref{thm:decide} to demonstrate the existence of infinite strength spherical designs.
Let \(\lambda\in\NN\) and let \(T_\lambda\) be a subset of \(\NN\) possessing the GCD property with period \(\lambda\). The nonzero \(0\)-\(1\) polynomials
of degree strictly less than \(\lambda\) can be enumerated as follows. For \(v=(a_0,\ldots,a_{\lambda-1})\in\mathbb F_2^\lambda\setminus\{0\}\), define
\[
  f_v\coloneq (1,x,\ldots,x^{\lambda-1})\cdot v
  =
  \sum_{i=0}^{\lambda-1}a_i x^i .
\]
Then
\[
  F_\lambda\coloneq
  \{f_v\mid v\in\mathbb F_2^\lambda\setminus\{0\}\}
\]
is precisely the set of all nonzero \(0\)-\(1\) polynomials of degree strictly less than \(\lambda\). Hence \(|F_\lambda|=2^\lambda-1\), and \Cref{thm:decide} implies that the existence of an infinite strength spherical design with period \(\lambda\) can be decided by testing at most \(2^{2^\lambda-1}\) subfamilies of \(F_\lambda\).
For example, we can prove the non-existence of such designs when $\lambda$ is sufficiently small by enumerating the polynomials, as demonstrated in the following proposition.
\begin{prop}\label{prop:six}
Let 
\[
  T=\{j\in\NN\mid \GCD(j,6)\in\{2,3\}\}.
\]
Then, for any $d\in\mathbb{N}$ and for any finite set $N\subset\mathbb{N}$, there does not exist an $X\subset\Sd^d$ such that $\Hst(X)=N\cup T$.
\end{prop}

\begin{proof}
    From \Cref{thm:hst}, such an $X\subset\Sd^d$ does not exist when $d\geq2$.
Therefore, we focus strictly on the existence of an $X$ in $\Sd^1$.
Let $v=(a_0,\ldots,a_5)\in\mathbb{F}_2^6\setminus\{0\}$ and define
    \[
        f_{v}\coloneq(1,x,x^2,x^3,x^4,x^5)\cdot v,
    \]
    where $\cdot$ denotes the inner product.
Among the nonzero \(0\)-\(1\) polynomials
\[
  \{f_v\mid v\in\mathbb F_2^6\setminus\{0\}\}
\]
of degree less than \(6\), the divisibility
\[
  \prod_{t\in\{2,3\}}\Phi_{6/t}
  =
  \Phi_3\Phi_2
  =
  x^3+2x^2+2x+1
  \mid f_v
\]
holds if and only if
\[
  v=(1,1,1,1,1,1).
\]
For this vector, we have
\[
  f_v=1+x+x^2+x^3+x^4+x^5
      =\Phi_2\Phi_3\Phi_6.
\]
Hence, if every \(f_j\) is divisible by \(\Phi_2\Phi_3\), then every \(f_j\)
is equal to \(1+x+\cdots+x^5\), and therefore
\[
  \GCD(f_1,\ldots,f_m,x^6-1)
\]
is divisible by \(\Phi_6\). In particular, it cannot be exactly
\[
  \prod_{t\in\{2,3\}}\Phi_{6/t}=\Phi_2\Phi_3.
\]
Therefore, there do not exist \(f_1,\ldots,f_m\) satisfying
\[
  \GCD(f_1,\ldots,f_m,x^6-1)=\prod_{t\in\{2,3\}}\Phi_{6/t}.
\]
    Then, from \Cref{thm:decide}, there does not exist an $X\subset\Sd^1$ such that $\Hst(X)=N\cup T$.
\end{proof}

Additionally, for specific infinite sets, we can demonstrate the existence of $X$ by explicitly providing the corresponding polynomials.
\begin{prop}\label{prop:single}
  Let $k$ and \textcolor{black}{$\lambda_p$ be any natural numbers satisfying $k\mid\lambda_p,\quad k<\lambda_p$, and define
  \[
    T_{k,\lambda_p}\coloneq\{j\in\NN\mid \GCD(j,\lambda_p)=k\},
\]
where $\lambda_p$ is the period of $T_{k,\lambda_p}$.
  Then, there exists an $X\subset\Sd^1$ such that $\Hst(X)=T_{k,\lambda_p}$.}
\end{prop}

\begin{proof}
\textcolor{black}{Put \(q=\lambda_p/k\), and write
\[
  q=p_1^{a_1}\cdots p_\ell^{a_\ell}.
\]
For each \(j\), 
\[
f_j(x)=1+x^{q/p_j}+x^{2q/p_j}+\cdots+x^{(p_j-1)q/p_j}.
\]
Equivalently, 
\begin{align*}
 f_j(x)&=\Phi_{p_j}\left(x^{q/p_j}\right)\\
    &=\prod_{\substack{c\mid \frac{q}{p_j} \\ \GCD(p_j,\frac{q}{cp_j})=1}}\Phi_{p_j c}\left(x\right)\\
    &=\prod_{c\mid \frac{q}{q_j}}\Phi_{q_j c}\left(x\right), 
\end{align*}
where $q_j=p_j^{a_j}$.
Note that in the second equation we used the following property (\cite{RN1944}, Exercise 2.57)\footnote{Exercise 2.57 only state this for the case where $m$ is prime. However, using induction, we can obtain the result when $m$ is not a prime number.}:
\[
    \Phi_n(x^m)=\prod_{\substack{c\mid m\\\GCD(n,\frac{m}{c})=1}}\Phi_{nc}(x).
\]
Therefore, 
\[
    \GCD(f_1,\ldots, f_\ell)=\Phi_q.
\]
This implies
\[
  \GCD(f_1,\ldots, f_\ell, x^{\lambda_p}-1)=\Phi_q.
\]
Since \(q=\lambda_p/k\), we have
\[
  \Phi_q=\Phi_{\lambda_p/k}.
\]
Thus the third condition of \Cref{thm:decide} is satisfied with
\[
  T_S=\{k\}.
\]
Therefore, by \Cref{thm:decide}, there exists an \(X\subset\Sd^1\) such that
\[
  \Hst(X)=T_{k,\lambda_p}.
\]
}
\end{proof}

\begin{rem}
\textcolor{black}{
When \(k=\lambda_p\), such an \(X\) does not exist, as we have seen in \Cref{rem:1}.}
\end{rem}

\begin{rem}
Note that the assumption in \Cref{prop:single} that \(\lambda_p\) is the period does
not lose generality.
Let \(k,\lambda\in\mathbb N\) satisfy \(k\mid\lambda\) and \(k<\lambda\), and put
\[
  T_{k,\lambda}:=\{j\in\mathbb N\mid \GCD(j,\lambda)=k\}.
\]
Write
\[
  \lambda=\prod_{p\mid\lambda}p^{a_p},
  \qquad
  k=\prod_{p\mid\lambda}p^{b_p},
\]
where \(0\leq b_p\leq a_p\). Define
\[
  \lambda_0
  :=
  \prod_{\substack{p\mid\lambda\\ b_p<a_p}}p^{b_p+1}
  \prod_{\substack{p\mid\lambda\\ b_p=a_p}}p^{b_p}.
\]
Then
\[
  T_{k,\lambda}
  =
  \{j\in\mathbb N\mid \GCD(j,\lambda_0)=k\}.
\]
Indeed, for each prime \(p\mid\lambda\), the condition
\[
  \GCD(j,\lambda)=k
\]
is equivalent to the following two cases:
\[
  v_p(j)=b_p \quad \text{if } b_p<a_p,
\]
and
\[
  v_p(j)\geq b_p \quad \text{if } b_p=a_p.
\]
These are exactly the conditions imposed by
\[
  \GCD(j,\lambda_0)=k.
\]
Moreover, the exponent of each prime in \(\lambda_0\) is minimal for detecting the corresponding condition. Indeed, if \(b_p<a_p\), then the condition is
\(v_p(j)=b_p\), so one must distinguish \(v_p(j)=b_p\) from \(v_p(j)>b_p\); hence the modulus must contain \(p^{b_p+1}\). If \(b_p=a_p\), then the condition is \(v_p(j)\ge b_p\), so one must distinguish \(v_p(j)\ge b_p\) from \(v_p(j)<b_p\); hence the modulus must contain \(p^{b_p}\). Therefore every period for \(T_{k,\lambda}\) is divisible by \(\lambda_0\). Since \(\lambda_0\) itself gives the above gcd description, \(\lambda_0\) is the period of \(T_{k,\lambda}\).

Therefore, if \(\lambda_p\) denotes the period of \(T_{k,\lambda}\), then
\[
  T_{k,\lambda}
  =
  \{j\in\mathbb N\mid \GCD(j,\lambda_p)=k\}.
\]
Thus the assumption in \Cref{prop:single} that \(\lambda_p\) is the period does
not lose generality.
\end{rem}

\begin{rem}
  We can also explicitly construct an $X\subset\Sd^1$ with $\Hst(X)=T_{k,\lambda_p}$.
Let $q=\frac{\lambda_p}{k}=p_1^{a_1}\cdots p_\ell^{a_\ell}$ and let $f_{j}$ be defined as in the proof of Proposition~\ref{prop:single}.
Then, define 
  \[
    X_{p_j,\lambda_p,k}\coloneqq X(f_{j},\zeta_{\lambda_p},e^{ji})
  \]
  and suppose 
  \[
    X=\bigsqcup_{j=1}^\ell X_{p_j,\lambda_p,k}.
\]
  Then, $\Hst(X)=T_{k,\lambda_p}$.
  For example, let $k=1$ and $\lambda_p=15$.
  Then, $q=3\cdot 5=p_1\cdot p_2$, and we have $f_{1}=1+x^{5}+x^{10}$ and $f_{2}=1+x^3+x^6+x^{9}+x^{12}$.
Therefore, $X_{3,15,1}$ is a triangle rotated by $e^i$, and $X_{5,15,1}$ is a pentagon rotated by $e^{2i}$.
From Theorem~\ref{thm:W}, \textcolor{black}{since $e^{ri}$ and $e^{2ri}$ are linearly independent over the algebraic numbers for each $r\in\mathbb{N}$,} $P_k(X_{3,15,1}\sqcup X_{5,15,1})=0$ holds if and only if both $P_k(X_{3,15,1})=0$ and $P_k(X_{5,15,1})=0$.
This implies 
  \[
    \Hst(X_{3,15,1}\sqcup X_{5,15,1})=\{j\in\NN\mid \GCD(j,15)=1\}.
\]
\end{rem}

\section{Conclusion}

In this paper, we focused on the existence of infinite strength spherical designs and the properties of their harmonic strength.
In Section~\ref{sec:inf}, we showed that infinite strength spherical designs with dimension $d \ge 2$ are always antipodal.
We obtained this result by utilizing an inequality for Jacobi polynomials, and we also deduced an upper bound for the maximum elements of the harmonic strength.
Section~\ref{sec:dim1} treats the case where the dimension is $1$, in which there are many infinite strength spherical designs that are not antipodal.
We demonstrated that even in this case, an infinite strength spherical design is always a cyclotomic design, which is a generalization of an antipodal set.
This proof relies on the Skolem-Mahler-Lech theorem concerning the zeros of complex linear recurrence sequences.
We also proved that the harmonic strength of a cyclotomic design possesses the weak GCD property.
In Section~\ref{sec:inverse}, we considered the inverse problem; that is, given an infinite set $T\subset\NN$, does there exist an $X\subset\Sd^1$ such that $\Hst(X)=T$?
To resolve this problem, we established Theorem~\ref{thm:decide} and showed that it can always be decided via a finite computation.
In our previous work, we established the existence of an $X\subset\Sd^1$ with $\Hst(X)=T$ for any finite subset $T\subset\NN$.
Therefore, the existence problem for spherical designs in $\Sd^1$ is completely settled.
However, the problem of optimality remains.
Let $T$ be a (possibly infinite) subset of $\NN$, and define the quantity $N(T,d)$ as follows:
\[
  N(T,d)\coloneq\min\{|X|\mid X\subset\Sd^d,\, \Hst(X)=T\},
\]
where we define $\min\emptyset=0$.
In general, Theorem~\ref{thm:decide} provides a finite computational method to determine whether $N(T,1)=0$ when $T$ is infinite and has the weak GCD property.
In fact, when $T$ has the GCD property, we can also compute $N(T,1)$ through finite calculations.
This is because there exist finite sets of polynomials corresponding to an $X$ with $\Hst(X)=T$.
However, this approach does not address the minimization of $N(T,1)$ when $T$ possesses only the weak GCD property.
\section*{Acknowledgment}
During this research, the authors received consistent and kind guidance from their supervisors, Professor Tsuyoshi Miezaki and Professor Akihiro Munemasa.
The many beneficial suggestions received from them and the lively discussions were essential for the completion of this research. Ryutaro Misawa was supported by JSPS KAKENHI Grant Number JP26KJ0570.
Yusaku Nishimura was supported by the Yoshida Scholarship Foundation through the Doctor 21 program and a Waseda Research Institute for Science and Engineering Grant-in-Aid for Young Scientists (Early Bird).

\bibliographystyle{plain}

\bibliography{sph}

\end{document}